\documentclass[reqno]{amsart}
\usepackage{amsmath,amsfonts,amssymb,amsthm,latexsym}
\usepackage{graphicx,subcaption}
\usepackage[hidelinks]{hyperref}
\numberwithin{equation}{section}
\usepackage[all,cmtip]{xy}

\newtheorem{theorem}{Theorem}[section]
\newtheorem{lemma}[theorem]{Lemma}
\newtheorem{corollary}[theorem]{Corollary}

\newtheorem{prop}[theorem]{Proposition}
\newtheorem{no}[theorem]{Notation}
\newtheorem{remark}[theorem]{Remark}

\theoremstyle{definition}
\newtheorem{definition}[theorem]{Definition}   

\renewcommand{\Im}{\operatorname{Im}}
\newcommand{\supp}{\operatorname{supp}}
\newcommand{\compconj}[1]{%
	\overline{#1}%
}
\renewcommand{\Im}{\operatorname{Im}}
\renewcommand{\Re}{\operatorname{Re}}
\newcommand{\tr}{\operatorname{tr}}

\begin{document}
	\title[Support of the Brown measure of a family]{Support of the Brown measure of a family of free multiplicative Brownian motions with non-negative initial condition}
	\author[S. Eaknipitsari]{Sorawit Eaknipitsari} 
	\address{University of Notre Dame\\
	Department of Mathematics\\
	255 Hurley Hall\\
	Notre Dame, IN, 46556\\U.S.A.} 
	\email{e.sorawit@gmail.com}
	\thanks{Eaknipitsari was supported by Development and Promotion of Science and Technology Talents Project (Royal Thai Government scholarship)}
	
	\author[B. Hall]{Brian C. Hall}
	\thanks{Hall was supported in part by a grant from the Simons Foundation}
	\address{University of Notre Dame\\
	Department of Mathematics\\
	255 Hurley Hall\\
	Notre Dame, IN, 46556\\U.S.A.} 
	\email{bhall@nd.edu}


	\keywords{Brown measure, free Brownian motion, Hamilton--Jacobi PDEs, random matrix theory}
	
	\date{\today}

	\begin{abstract}
	We consider a family $b_{s,\tau}$ of free multiplicative Brownian motions labeled by a real variance parameter $s$ and a complex covariance parameter 
	$\tau$. We then consider the element $xb_{s,\tau}$, where $x$ is non-negative and freely independent of $b_{s,\tau}$. Our goal is to identify the support of the Brown measure of $xb_{s,\tau}$. 
	
	In the case $\tau =s$, we identify a region $\Sigma_s$ such that the Brown measure is vanishing outside of $\overline{\Sigma}_s$ except possibly at the origin. For general values of $\tau$, we construct a map $f_{s-\tau}$ and define $D_{s,\tau}$ as the complement of $f_{s-\tau}(\overline{\Sigma}_s^c)$. Then the Brown measure is zero outside $D_{s,\tau}$ except possibly at the origin.
	
	The proof of these results is based on a two-stage PDE analysis, using one PDE (following the work of Driver, Hall, and Kemp) for the case $\tau=s$ and a different PDE (following the work of Hall and Ho) to deform the $\tau=s$ case to general values of $\tau$.
	\end{abstract}
	
	\maketitle
	\tableofcontents
	\section{Introduction}
	\subsection{Free multiplicative Brownian motions and their Brown measures}
	
In \cite{bianejfa}, Biane introduced the \textit{free multiplicative Brownian motion} $b_t$ as an element of a tracial von Neumann algebra. As conjectured by Biane and proved by Kemp \cite{kemp}, $b_t$ is the limit in $*$-distribution of the standard Brownian motion in the general linear group $GL(N;\mathbb C)$ as $N\rightarrow\infty$. (See \cite{banna} for a stronger version of Kemp's result.) For a fixed $t>0$, the element $b_t$ can be approximated by an element of the form 
\begin{equation}\label{approxBrownian}
b_t \sim \left(I+\sqrt{\frac{t}{k}}z_1\right)\dots \left(I+\sqrt{\frac{t}{k}}z_k\right),
\end{equation}
where $z_1,\dots,z_k$ are freely independent circular elements and $k$ is large. More precisely, $b_t$ is defined as the solution of a \emph{free} It\^o stochastic differential equation, as in Section 2.2 and then Theorem 1.14 in \cite{sevenAuthors} shows that (\ref{approxBrownian}) approximates this solution. 
	
	There is also a ``three-parameter'' generalization $b_{s,\tau}$ of $b_t$, labeled by a real variance parameter $s$ and a complex covariance parameter $\tau$. (The three parameters are $s$ and the real and imaginary parts of $\tau$.) The original free multiplicative Brownian motion $b_t$ corresponds to the case $s=\tau=t.$ The case in which $\tau=0$ gives Biane's free \textit{unitary} Brownian motion $u_s=b_{s,0}$.
	
		In the case that $\tau$ is real, the \textit{support} of the Brown measure of the $b_{s,\tau}$ was computed by Hall and Kemp \cite{hk}. Then Driver, Hall and Kemp \cite{dhk} computed the actual Brown measure of $b_t$ (not just its support). This result was then extended by Ho and Zhong \cite{hz}, who computed the Brown measure of $ub_t$, where $u$ is a unitary ``initial condition,'' assumed to be freely independent of $b_t$. Hall and Ho \cite{hho} then computed the Brown measure of $ub_{s,\tau}$ for arbitrary $s$ and $\tau$. 
		
		Finally, Demni and Hamdi \cite{dem} computed the support of the Brown measure of $pb_{s,0}$ where $p$ is a projection that is freely independent of $b_{s,0}$. Although Demni and Hamdi extend many of the techniques used in \cite{dhk,hz,hho} to their setting, the fact that the initial condition $p$ is not unitary causes difficult technical issues that prevent them from computing the actual Brown measure of $pb_{s,0}$. 
	
	In this paper, we study $xb_{s,\tau}$, where the initial condition $x$ is taken to be non-negative and freely independent of $b_{s,\tau}$. We will find a certain closed subset $D_{s,\tau}$ with the property that the Brown measure of $xb_{s,\tau}$ is zero outside $D_{s,\tau}$, except possibly at the origin.  Simulations and analogous results for other cases strongly suggest that the closed support of the Brown measure of $xb_{s,\tau}$ is precisely $D_{s,\tau}$ (or $D_{s,\tau}\cup \{0\}$).
	
	One important aspect of the problem is to understand how the domains $D_{s,\tau}$ vary with respect to $\tau$ with $s$ fixed. (Compare Definition 2.5 and Section 3 in \cite{hho} in the case of a unitary initial condition.) For each $s$ and $\tau$, we will construct a holomorphic map $f_{s-\tau}$ defined on the complement of $D_{s,s}$. We will show that this map is injective and tends to infinity at infinity. Then the \emph{complement} of $D_{s,\tau}$ will be the image of the \emph{complement} of $D_{s,s}$ under $f_{s-\tau}$. Thus, all the domains with a fixed value of $s$ can be related to the domain $D_{s,s}$ by means of $f_{s-\tau}$. It then follows that the topology of the complement of $D_{s,\tau}$ is the same for all $\tau$ with $s$ fixed. By contrast, Figure \ref{fig:first} shows that the topology of the complement of $D_{s,\tau}$ can change when $s$ changes (in this case, with $\tau=s$).

	When $\tau=0$ and $x$ is a projection, our result reduces to the one obtained by Demni and Hamdi. Thus, our work generalizes \cite{dem} by allowing arbitrary values of $\tau$ and arbitrary non-negative initial conditions. The difficulties in computing the Brown measure in the setting of Demni and Hamdi persist in our setting and we do not address that problem here. See Remark \ref{hard.remark} for an indication of why the case of a non-negative initial condition is harder than a unitary initial condition.
	
	\subsection{The support of $b_t$ with non-negative initial condition}\label{btOutline.sec}
	
	In this subsection, we briefly describe how our results are obtained in the case $\tau=s$. In the next subsection, we describe how the case of general $\tau$ is reduced to the case $\tau=s.$ Let $ \mu$ be the law (or spectral distribution) of the non-negative initial condition $x$. That is, $\mu$ is the unique probability measure on $[0, \infty)$ satisfying
	\begin{equation}\label{mudef}\int_{0}^{\infty}t^k \,d\mu(t) = \tr(x^k)\end{equation}
	for all $k = 0,1,2,\dots$, where $\tr$ is the trace on the relevant von Neumann algebra.
	
	We next define the regularized $\log$ potential function $S$ as
	\begin{align*}
		S(t,\lambda,\varepsilon) = \tr\left[\log\left((xb_t-\lambda)^\ast(xb_t -\lambda)+\varepsilon\right)\right], \quad \varepsilon > 0.		
	\end{align*}  
	and its limit as $\varepsilon \rightarrow 0^+$
	\[s_t(\lambda) = \lim_{\varepsilon \rightarrow 0^+} S(t,\lambda,\varepsilon).\]
	Then, following Brown \cite{brown} (see also Chapter 11 of the monograph of Mingo and Speicher \cite{mingo}), the Brown measure $\mu_t$ of $hb_t$ is the distributional Laplacian
	\[\mu_t = \frac{1}{4\pi}\Delta s_t(\lambda).\]
	According to Proposition \ref{limS.prop}, the function $s$ is in $L^1_{\mathrm{loc}}$ and is subharmonic.
	
	The function $S$ satisfies the following PDE, obtained similarly as in \cite{hz}, in logarithmic polar coordinates:
	\begin{equation}\label{firstPDE}
		\frac{\partial S}{\partial t} = \varepsilon\frac{\partial S}{\partial \varepsilon}\left(1+(|\lambda|^2-\varepsilon)\frac{\partial S}{\partial \varepsilon} - \frac{\partial S}{\partial \rho}\right), \quad \lambda = e^\rho e^{i\theta} = re^{i\theta}
	\end{equation}
	with initial condition
	\begin{equation}\label{initialCond}S(0,\lambda,\varepsilon) = \tr[\log(x-\lambda)^\ast(x-\lambda)+\varepsilon] = \int_0^{\infty} \log(|\xi-\lambda|^2 +\varepsilon) \,d\mu(\xi). \end{equation}
	
	Following the PDE method given by \cite{dhk}, we consider the following Hamiltonian, obtained by replacing each derivative on the right-hand side of \eqref{firstPDE} by a ``momentum'' variable, with an overall minus sign:
	\begin{equation}\label{Hdef}
		H(\rho,\theta,\varepsilon,p_\rho,p_\theta,p_\varepsilon) = -\varepsilon p_{\varepsilon}(1+(r^2 - \varepsilon)p_{\varepsilon} - p_{\rho}).
	\end{equation}
		We then consider Hamilton's equations, given as
	\begin{equation}\label{HamEq}\frac{d\rho}{dt} = \frac{\partial H}{\partial p_\rho},\quad \frac{dp_\rho}{dt} = -\frac{\partial H}{\partial \rho},\end{equation}
	and similarly for other pairs of variables.
	Given initial conditions for the ``position'' variables:
	\[\rho(0) = \rho_0, \quad r(0) = r_0, \quad \varepsilon(0) = \varepsilon_0,\]
	we take the initial conditions for the momentum variables to be
	\begin{align*}
		p_{\rho,0}(\lambda_0, \varepsilon_0) = \frac{\partial S(0,\lambda_0, \varepsilon_0)}{\partial \rho_0},\\ p_{\theta}(\lambda_0, \varepsilon_0) = \frac{\partial S(0,\lambda_0, \varepsilon_0)}{\partial \theta}, \\ p_{0}(\lambda_0, \varepsilon_0) = \frac{\partial S(0,\lambda_0, \varepsilon_0)}{\partial \varepsilon_0}.
	\end{align*}	
	
	\begin{no}
	In the above formulas, we follow \cite[Eq. (5.4)]{dhk} in using the expression $p_0$ (as opposed to $p_{\varepsilon,0}$) to denote the value of $p_{\varepsilon}$ at $t=0$. This notation is consistent the $k=0$ case of the notation $p_k$ introduced below in (\ref{pkdef}).
	\end{no}
	
	The solution $S$ to the PDE (\ref{firstPDE}) then satisfies the first Hamilton--Jacobi formula:
	\begin{equation}\label{pdecon}
		S(t,\lambda(t),\varepsilon(t)) = S(0,\lambda_0,\varepsilon_0) - H_0 t + \log(|\lambda(t)| )- \log(|\lambda_0|),
	\end{equation}
	\textit{provided} the solution to the Hamilton's equation exists up to time $t$. (See \cite[Eqs. (5.20) and (5.21)]{dhk}.)
	Since our aim is to compute
	\[s_t(\lambda) = S(t,\lambda,0),\] 
	we want to choose good initial values $\varepsilon_0$ and $\lambda_0$ so that 
	\begin{equation}\label{goal}\lambda(t) = \lambda \text{ and }\varepsilon(t) = 0.\end{equation}
	The next result says that we can achieve the goal in \eqref{goal} by taking $\varepsilon_0$ approaching 0 and taking $\lambda_0=\lambda$, \textit{provided} that the solution of Hamilton's equations exists up to time $t$.
	\begin{lemma}\label{smalleps.lem} 
	Assume that $\lambda_0$ is outside the support of $\mu$. Then in the limit as $\varepsilon_0\rightarrow 0$, we have $\varepsilon(t) \equiv 0$ and $\lambda(t) \equiv \lambda$, for as long as the solution to Hamilton's equations exists. 
	\end{lemma}
	The proof is given on p. \pageref{limitproof}. If the lemma applies, the Hamilton--Jacobi formula (\ref{pdecon}) becomes
	\begin{equation*} S(t,\lambda,0) = S(0, \lambda_0=\lambda,\varepsilon_0=0) - H_0 t. \end{equation*}
	Furthermore, when $\varepsilon_0=0$, we compute from (\ref{Hdef}) that $H_0=0$, so that we obtain
	\begin{equation}\label{limitHJ} S(t,\lambda,0) = S(0, \lambda,0). \end{equation}
	We emphasize, however, that this conclusion is valid only if the lifetime of the solution of Hamilton's equations is greater than $t$ when $\varepsilon_0 \rightarrow 0$.
	
	Now, we will compute that the limit of the lifetime of solutions to Hamilton's equation, in the limit when $\varepsilon_0 \rightarrow 0$, as
	\[T(\lambda) = \frac{1}{{\tilde p}_2 - {\tilde p}_0r^2}\log\left(\frac{{\tilde p}_2}{{\tilde p}_0r^2}\right)\]
	where
	\[{\tilde p}_k = \int_0^{\infty} \frac{\xi^k}{|\xi-\lambda|^2} \,d\mu(\xi).\] Thus, we define a domain $\Sigma_t$ as follows:
	\[\Sigma_t = \{\lambda \,|\, T(\lambda) < t\}.\]
	
	\begin{figure}[t]
		\begin{subfigure}{0.4\textwidth}
			\centering
			\includegraphics[height=1.6in]{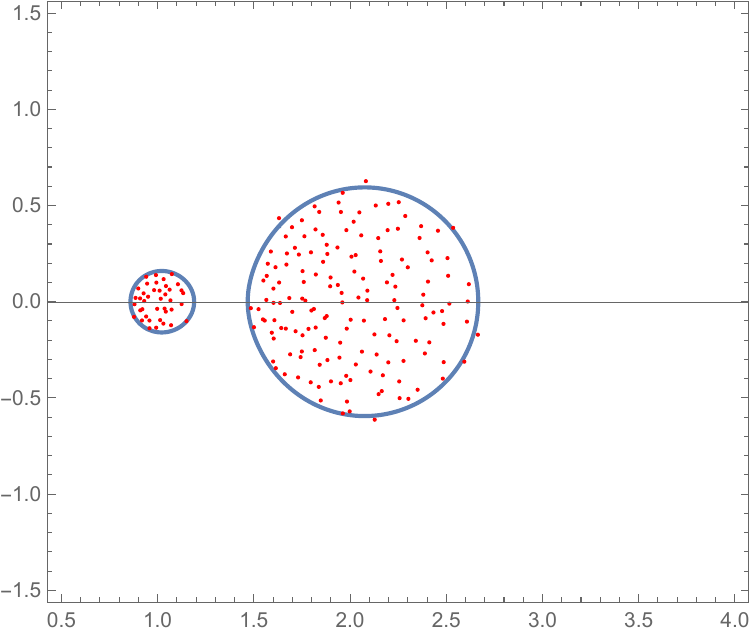}
			\caption{$t = 0.1$}
		\end{subfigure}
		\quad		\begin{subfigure}{0.4\textwidth}
			\centering
			\includegraphics[height=1.6in]{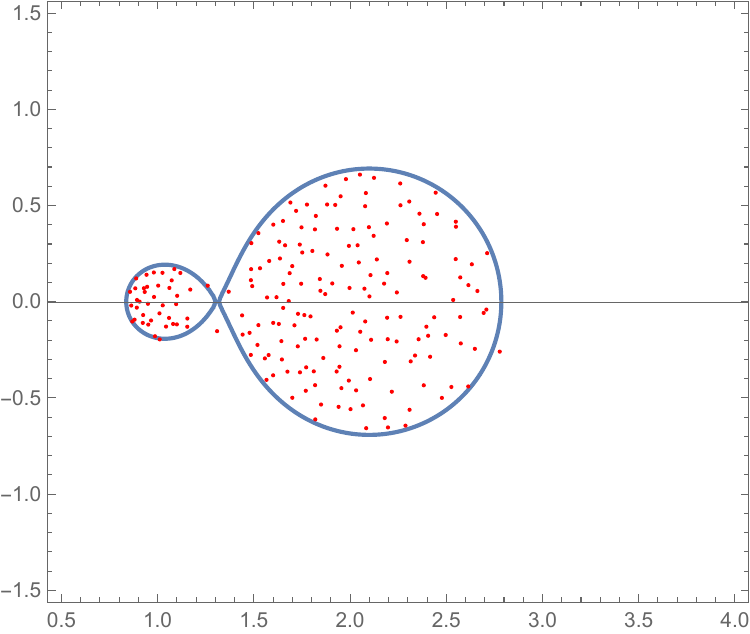}
			\caption{$t = 0.1319$}
		\end{subfigure}\\
		
		\begin{subfigure}{0.4\textwidth}
			\centering
			\includegraphics[height=1.6in]{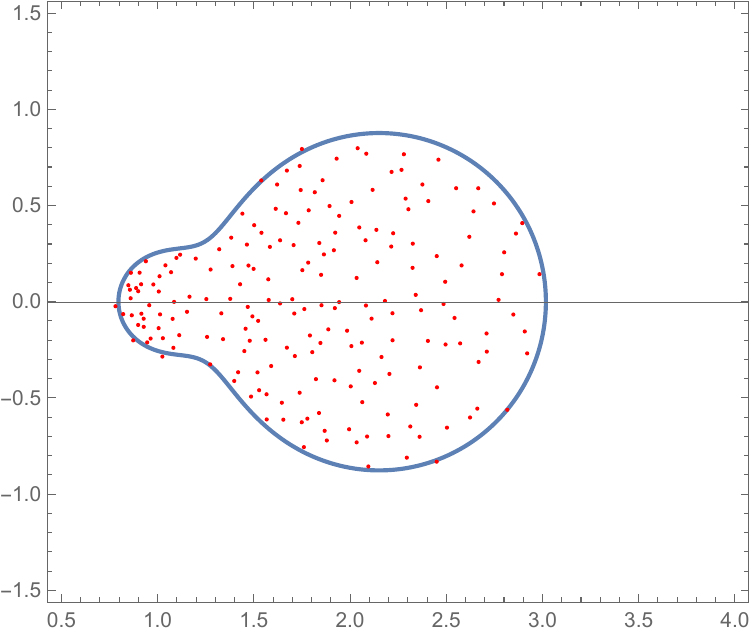}
			\caption{$t = 0.2$}
		\end{subfigure}
		\quad		\begin{subfigure}{0.4\textwidth}
			\centering
			\includegraphics[height=1.6in]{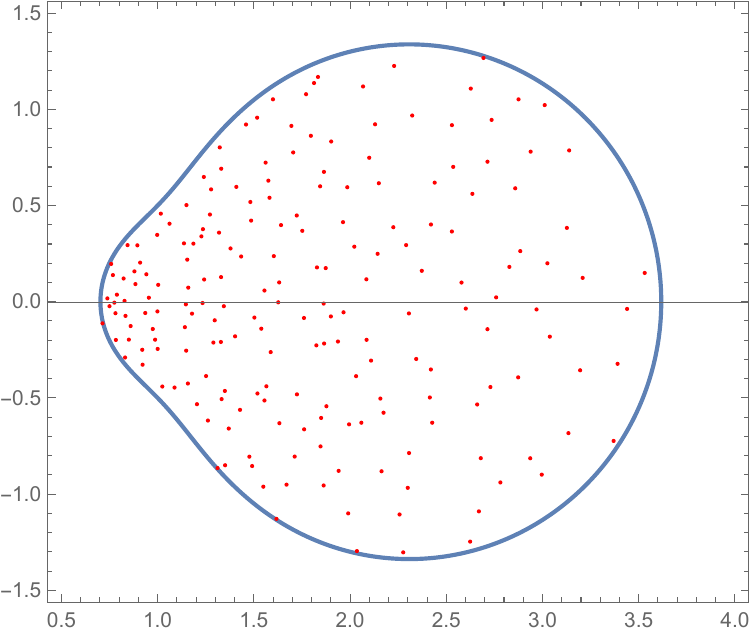}
			\caption{$t = 0.4$}
		\end{subfigure}
		
		\caption{The domain $\overline{\Sigma}_t$ with the eigenvalues (red dots) of a random matrix approximation to $xb_t$, in the case $\mu =\frac{1}{5}\delta_{1} +\frac{4}{5}\delta_2$}
		\label{fig:first}
	\end{figure}
	
	If we insert the initial condition (\ref{initialCond}) (at $\varepsilon=0$) into the formula (\ref{limitHJ}), we obtain the following result, whose proof has been outlined above.
	\begin{theorem}[Free multiplicative Brownian motion with non-negative initial condition]
		For all $(t,\lambda)$ with $\lambda$ outside $\compconj{\Sigma}_t$, we have
		$$ s_t(\lambda) := \lim_{\varepsilon_0 \rightarrow 0^+} S(t,\lambda,\varepsilon)= \int_0^{\infty}\log|\xi-\lambda|^2\,d\mu(\xi).$$
		Since, as we will show, the closed support of $\mu$ is contained in $\overline{\Sigma}_t\cup \{0\}$, it follows that the Brown measure $\mu_t$ is 
		zero outside of $\overline{\Sigma}_t$, except possibly at the origin. 
	\end{theorem}
	\noindent See Figure \ref{fig:first} for the domain $\overline{\Sigma}_t$ plotted with the eigenvalues (red dots) of a random matrix approximation to $xb_t$, in the case $\mu =\frac{1}{5}\delta_{1} +\frac{4}{5}\delta_2$.

\subsection{The case of arbitrary $\tau$}	

We now consider a family $b_{s,\tau}$ of free multiplicative Brownian motions, labeled by a positive real number $s$ and a complex number $\tau$ satisfying
\[
|\tau-s|\leq s.
\]
These were introduced by Ho \cite{hoSBT} when $\tau$ is real and by Hall and Ho \cite{hho} when $\tau$ is complex. (See Section \ref{brownians.sec} for details.) When $\tau=s$, the Brownian motion $b_{s,s}$ has the same $*$-distribution as $b_s$ and when $\tau=0$, the Brownian motion $b_{s,0}$ has the same $*$-distribution as Biane's free unitary Brownian motion $u_s$.

When $\tau$ is real, the support of the Brown measure of $b_{s,\tau}$ was computed by Hall and Kemp \cite{hk}, using the large-$N$ Segal--Bargmann transform developed by Driver--Hall--Kemp \cite{dhk_largeN} and Ho \cite{hoSBT}. The Brown measure of $ub_{s,\tau}$ when $u$ is unitary and freely independent of $b_{s,\tau}$ was computed in \cite{hho}. In this paper, we determine the support of the Brown measure of $xb_{s,\tau}$ when $x$ is non-negative and freely independent of $b_{s,\tau}$. 

To attack the problem for arbitrary $\tau$, we will show that the regularized log potential of $xb_{s,\tau}$ satisfies a PDE with respect to $\tau$ with $s$ fixed. We solve this PDE using as our initial condition the case $\tau=s$---which we have already analyzed, as in the previous subsection. To solve the PDE, we again use the Hamilton--Jacobi method and we again put $\varepsilon_0$ equal to 0. With $\varepsilon_0=0$, we again find that $\varepsilon(t)$ is identically zero---but this time $\lambda(t)$ is not constant. Rather, for $\lambda_0$ outside $\overline{\Sigma}_t$, we find that with $\varepsilon_0=0$, we have
\begin{equation}
\lambda(t)=f_{s-\tau}(\lambda_0)
\end{equation}
where $f_{s-\tau}$ is a holomorphic function given by
	\[f_{s-\tau}(z) = z\exp\left[\frac{s-\tau}{2}\int_0^{\infty} \frac{\xi+z}{\xi-z}\,d\mu(\xi)\right].\]

The Hamilton--Jacobi method will then give a formula for the log potential of the Brown measure of $xb_{s,\tau}$, valid at any nonzero point $\lambda$ of the form 
$\lambda=f_{s-\tau}(\lambda_0)$, with $\lambda_0$ outside $\overline{\Sigma}_t$. This formula will show that the Brown measure of $xb_{s,\tau}$ is zero near $\lambda$. 

We summarize the preceding discussion with the following definition and theorem.
	\begin{definition}
		For all $s>0$ and $\tau \in \mathbb{C}$ such that $|\tau-s| \leq s$, we define a closed domain $D_{s,\tau}$ characterized by
		\[D_{s,\tau}^c = f_{s-\tau}(\overline{\Sigma}_s^c).\]
	\end{definition}

That is to say, the complement of $D_{s,\tau}$ is the image of the complement of $\overline{\Sigma}_s$ under $f_{s-\tau}$. See Figure \ref{fig:second}. When $\tau=s$, we have that $f_{s-\tau}(z)=z$, so that $D_{s,s}=\overline\Sigma_s$.

	\begin{theorem}
		For all $s>0$ and $\tau \in \mathbb{C}$ such that $|\tau-s| \leq s$, the Brown measure of $xb_{s,\tau}$ is zero outside $D_{s,\tau}$, except possibly at the origin.
	\end{theorem}
	
	When the origin is not in $D_{s,\tau}$, we will, in addition, show that the mass of the Brown measure at the origin equals the mass of the measure $\mu$ at the origin.	
	
	\begin{figure}[t]
		\begin{subfigure}{0.4\textwidth}
			\centering
			\includegraphics[height=1.7in]{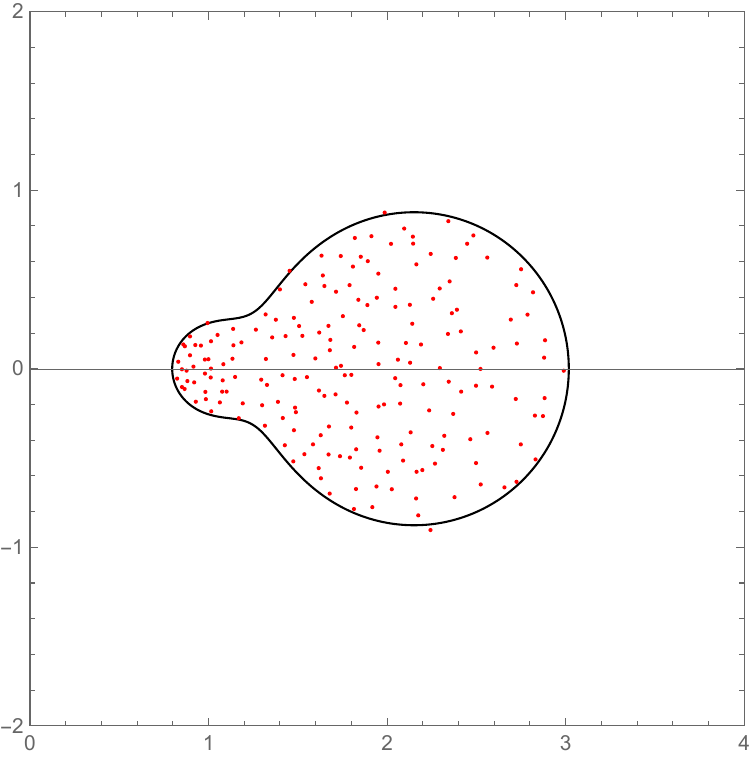}
			\caption{$\tau =s$}
		\end{subfigure}
		\quad		\begin{subfigure}{0.4\textwidth}
			\centering
			\includegraphics[height=1.7in]{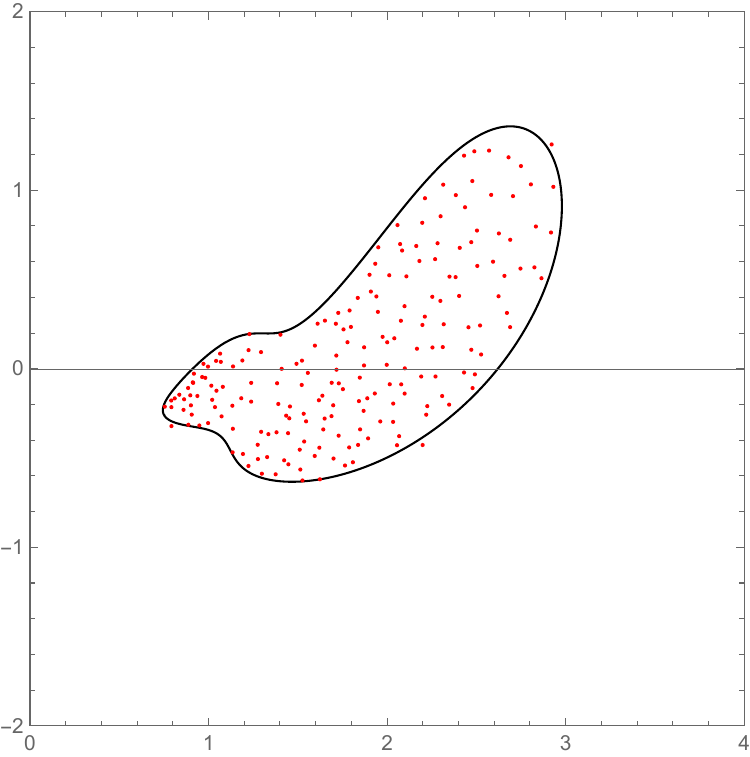}
			\caption{$\tau =s + 0.1i$}
		\end{subfigure}\\
		
		\begin{subfigure}{0.4\textwidth}
			\centering
			\includegraphics[height=1.7in]{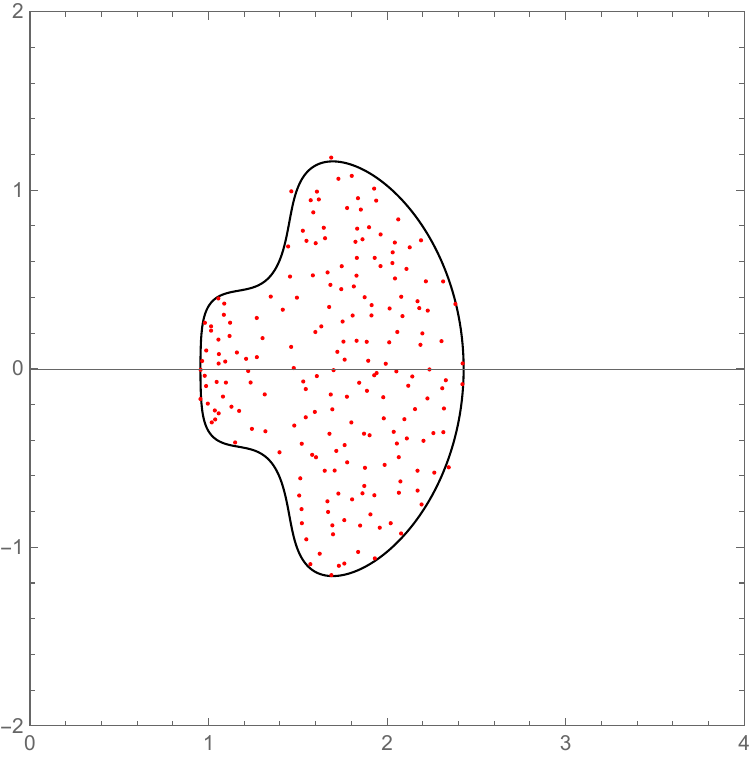}
			\caption{$\tau =s-0.1$}
		\end{subfigure}
		\quad		\begin{subfigure}{0.4\textwidth}
			\centering
			\includegraphics[height=1.7in]{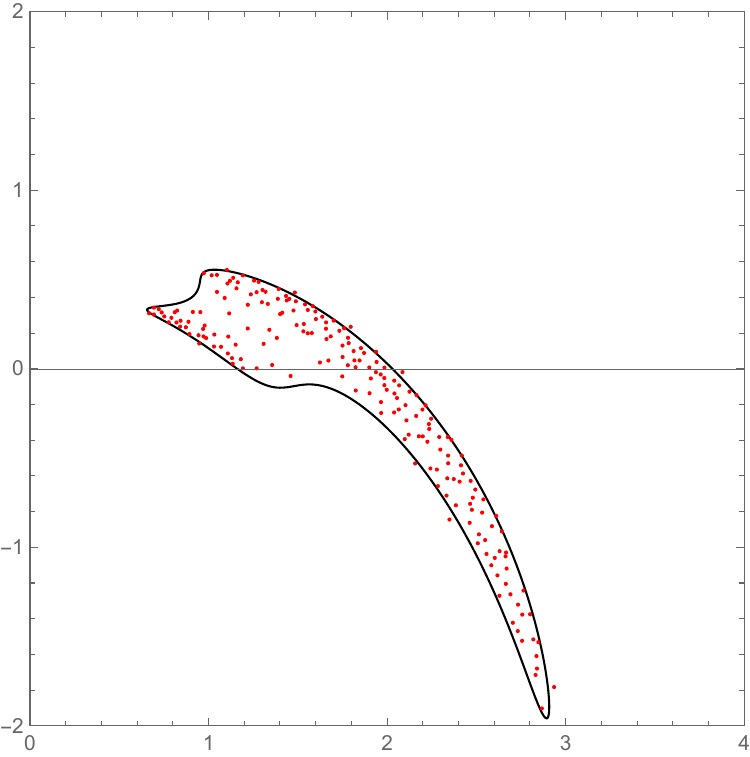}
			\caption{$\tau =s-0.19i$}
		\end{subfigure}
		
		\caption{The domain $D_{s,\tau}$ along with the eigenvalues (red dots) of a random matrix approximation to $xb_{s,\tau}$, with $s=0.2$ and $\mu =\frac{1}{5}\delta_{1} +\frac{4}{5}\delta_2$.}
		\label{fig:second}
	\end{figure}
	
	\begin{remark}
	Although this paper uses the PDE method introduced in \cite{dhk}, one could attempt to follow the method of Hall and Kemp in \cite{hk}, which computes the support of the Brown measure of $b_{s,\tau}$ (in the case $\tau$ is real). The paper \cite{hk} makes use of the free Segal--Bargmann transform introduced by Biane in \cite{bianejfa} and extended by Ho \cite{hoSBT}. This transform is the large-$N$ limit of the Segal--Bargmann transform of the second author \cite{hall1} for the unitary group $U(N)$. (See also \cite{newform}, \cite{dhk_largeN}, and \cite{chan}.)
	
	 One could then attempt to incorporate the non-negative element $x$ into the analysis of \cite{hk}. This would require extending the results of \cite{hall1} and \cite{bianejfa,hoSBT} to handle arbitrary (not necessarily unitary) initial conditions. Even if this extension were successful, one would still have to understand the support of the Brown measure of $xu_t$, where $u_t$ is the free unitary Brownian motion, as the starting point for the analysis. But the only known method for computing this support is the PDE method, either in the form used in \cite{dem} or in the form used in the present paper. At that point, it makes more sense to simply use the PDE method throughout. 
	\end{remark}
	
	\section{Preliminaries}
	
	\subsection{Free Probability}
	A \textit{tracial von Neumann algebra} is a pair $(\mathcal{A},\tr)$, where $\mathcal{A}$ is a von Neumann algebra and $\tr:\mathcal{A}\rightarrow\mathbb C$ is a faithful, normal, tracial state on $\mathcal{A}$. Here ``tracial'' means that $\tr(ab)=\tr(ba)$, ``faithful'' means that $\tr(a^*a)>0$ for all nonzero $a$, and ``normal'' means that $\tr$ is continuous with respect to the weak operator topology.	The elements in $\mathcal{A}$ are called (noncommutative) random variables.
	
	Unital $\ast$-subalgebras $\mathcal{A}_1,\dots,\mathcal{A}_n \subset \mathcal{A}$ are said to be \emph{freely independent} if given any $i_1,\dots, i_m \in \{1,\dots,n\}$ with $i_k \neq i_{k+1}$ and $a_{i_k} \in \mathcal{A}_{i_k}$ such that $\tr(a_{i_k}) = 0$ for all $1\leq k \leq m$, we have $\tr(a_{i_1}\cdots a_{i_m}) =0.$ Moreover, random variables $a_1,\dots,a_m$ are said to be freely independent if the unital $\ast$-subalgebras generated by them are freely independent.
	
	For any self-adjoint random variable $a\in \mathcal{A}$, the law or the \emph{distribution} of $a$ is the unique compactly supported probability measure on $\mathbb{R}$ such that for any bounded continuous function $f$ on $\mathbb{R}$, we have 
	\begin{equation}\label{lawdef} \int f\, d\mu = \tr(f(a)).\end{equation}

	\subsection{Free Brownian motions}\label{brownians.sec}
	In free probability, the semicircular law plays a role similar to the Gaussian distribution in classical probability. The semicircular law $\sigma_t$ with variance $t$ is the probability measure supported in $[-2\sqrt{t},2\sqrt{t}]$ with density there given by
	$$ d \sigma_t(\xi) = \frac{1}{2\pi t}\sqrt{4t-\xi^2}.$$
	\begin{definition}
		A \emph{free semicircular Brownian motion} $s_t$ in a tracial von Neumann algebra $(\mathcal{A},\tr)$ is a weakly continuous free stochastic process $(s_t)_{t\geq 0}$ with freely independent increments and such that the law of $s_{t_2}-s_{t_1}$ is semicircular with variance $t_2-t_1$ for all $0<t_1<t_2$.  A \emph{free circular Brownian motion} $c_t$ has the form $\frac{1}{\sqrt{2}}(s_t+is_t')$ where $s_t$ and $s_t'$ are two freely independent free semicircular Brownian motions.
	\end{definition}
	\begin{definition}\label{def b_t}
		The \emph{free multiplicative Brownian motion} $b_t$ is the solution of the free It\^o stochastic differential equation
		$$ db_t =b_t\,dc_t, \quad b_0 =I,$$
		where $c_t$ is a free circular Brownian motion.
	
	\end{definition}
	We refer to work of Biane and Speicher \cite{BianeSpeicher} and Nikitopoulos \cite{vaki} for information about free stochastic calculus and to \cite[Section 4.2.1]{bianejfa} for information about the free multiplicative Brownian motion (denoted there as $\Lambda_t$). According to \cite{bianejfa}, $b_t$ is invertible for all $t$. Then the \emph{right} increments of $b_t$ are freely independent. That is for every $0<t_1 <\dots<t_n$ in $[0, \infty)$, the random variables
	$$b_{t_1},b_{t_1}^{-1}b_{t_2}, \dots, b_{t_{n-1}}^{-1}b_{t_n}$$
	are freely independent. 
	
	Now, to define \emph{free multiplicative $(s,\tau)$-Brownian motion}, we introduce a \emph{rotated elliptic element} as follows.
	
	\begin{definition}
		A \emph{rotated elliptic element} is an element $Z$ of the following form
		\[ Z = e^{i\theta}\left(aX+ibY\right)\]
		where $X$ and $Y$ are freely independent semicircular elements, $a,b,$ and $\theta$ are real numbers, and we assume $a$ and $b$ are not both zero.
	\end{definition}
	
	As in Section 2.1 in \cite{hho}, we then parameterize rotated elliptic elements by two parameters: a positive variance parameter $s$ and a complex covariance parameter $\tau$ defined by 
	\begin{align*}
		s &= \tr[Z^\ast Z]\\
		\tau &= \tr[Z^\ast Z] - \tr[Z^2].
	\end{align*}
	By applying the Cauchy--Schwarz inequality to the inner product $\tr(A^\ast B)$ we see that any rotated elliptic element satisfies
	\begin{equation}\label{tauIneq}|\tau -s| \leq s.\end{equation}
	Conversely, if $s$ and $\tau$ satisfy (\ref{tauIneq}), we can construct a rotated elliptic element with those parameters by choosing $a, b$ and $\theta$ as
	\begin{align*}
		a &= \sqrt{\frac{1}{2}(s+|\tau-s|)}\\
		b &= \sqrt{\frac{1}{2}(s-|\tau-s|)}\\
		\theta &= \frac{1}{2}\arg(s-\tau).
	\end{align*}
		Note that if $\tau = s$, then we have $a=b$ and $Z$ is a circular element with variance $s$, having $\ast$-distribution independent of $\theta$.

		A \emph{free additive $(s,\tau)$-Brownian motion} is a continuous process $w_{s,\tau}(r)$ with $w_{s,\tau}(0) = 0$ having freely independent increments such that for all $r_2 >r_1$,
		\[\frac{w_{s,\tau}(r_2)-w_{s,\tau}(r_1)}{\sqrt{r_2-r_1}}\]
		is a rotated elliptic element with parameter $s$ and $\tau$. We can construct such an element as
		\[ w_{s,\tau}(r)= e^{i\theta}(aX_r+ibY_r)\]
		where $X_r$ and $Y_r$ are freely independent semicircular Brownian motion and $a,b,$ and $\theta$ are chosen as above.
		
		\begin{definition}\label{def b_st}
		A \emph{free multiplicative $(s,\tau)$-Brownian motion} $b_{s,\tau}(r)$ is the solution of the free stochastic differential equation 
	\begin{equation}\label{sde}
		db_{s,\tau}(r)= b_{s,\tau}(r)\left(i\, dw_{s,\tau}(r)-\frac{1}{2}(s-\tau)\,dr\right),
	\end{equation}
		with $$b_{s,\tau}(0) = 1.$$ 
		\end{definition}
		The $dr$ term in (\ref{sde}) is an It\^o correction. 
		Since $w_{s,\tau}(r)$ and $w_{rs,r\tau}(1)$ have the same $\ast$-distribution, it follows that  $b_{s,\tau}(r)$ and $b_{rs,r\tau}(1)$ also have the same $\ast$-distribution. Thus, without loss of generality, we may assume that $r=1$ and use the notation
		$$b_{s,\tau}:= b_{s,\tau}(1).$$
	
	When $\tau=s$, the It\^o correction vanishes and we find that 
		\[
		b_{s,s}=b_s.
		\]
	Furthermore, when $\tau =0$ we have that $a = s, b=0=\theta$ and $w_{s,0} = sX_r$. Then (\ref{sde}) becomes
	\[db_{s,0}(r)= sb_{s,0}(r)\left(i\, dX_r-\frac{1}{2}\,dr\right).\]
	This equation is the same SDE for the free unitary Brownian motion $U_s$ considered by Biane in Section 2.3 of \cite{bia}. Therefore we can identify $b_{s,0}$ with $U_s$ in \cite{bia}.
	\begin{remark}
	According to Proposition 6.10 in \cite{banna}, the $*$-distribution of $b_{s,\tau}$ is unchanged if we reverse the order of the factors on the right-hand side of (\ref{sde}), that is, putting the increments on the left instead of the right of $b_{s,\tau}$. This result is proved by using a matrix approximation to $b_{s,\tau}$ and appealing to a result of Driver \cite[Theorem 2.7]{driverSBT}. So far as we know, a general proof of this result in the free setting has not appeared. In the case $\tau=s$, however, the result follows from \cite[Theorem 1.14]{sevenAuthors}, using a discrete-time approximation to the SDE defining $b_{s,s}=b_s$.
	\end{remark}
		\subsection{The Brown measure}
	For a normal random variable $x \in \mathcal{A}$, we can define the law or distribution of $x$ as a compactly supported probability measure on the plane as follows. The spectral theorem (e.g., \cite[Section 10.3]{qmbook}) associates to $x$ a unique projection-valued measure $\nu^x$, supported on the spectrum of $x$, such that
	$$ x = \int \lambda\, d\nu^x(\lambda).$$
	Then the law $\mu_x$ of $x$ can be defined as 
	$$\mu_x(A) = \tr(\nu^x(A)),$$
	for each Borel set $A$.
	
	If, however, $x$ is not normal, the spectral theorem does not apply. Nevertheless, a candidate for the distribution of a non-normal operator was introduced by Brown \cite{brown}.  For an operator $a$, we use the standard notation $|a|$ for the non-negative square root of $a^*a$.
	\begin{definition}
	For any $x\in\mathcal A$, we define a function $S:\mathbb C\times (0, \infty)$ by
	\begin{equation*}
	S(\lambda,\varepsilon)=\operatorname{tr}[\log(|x-\lambda|^2+\varepsilon)]
	\end{equation*}
	and a function $s:\mathbb C\rightarrow [-\infty, \infty)$ by
	\begin{equation*}
	s(\lambda)=\lim_{\varepsilon\rightarrow 0^+}S(\lambda,\varepsilon).
	\end{equation*}
	\end{definition}
	
	The following result explains the sense in which the limit defining $s$ should be understood. Although it is possible the $L^1_{\mathrm{loc}}$ convergence is known, we have not seen such a result in the literature. 
	
	\begin{prop}\label{limS.prop}
Let $x$ be an element of a tracial von Neumann algebra and define a function
$S:\mathbb{C}\times(0, \infty)\rightarrow\mathbb{R}$ by%
\[
S(\lambda,\varepsilon)=\mathrm{tr}[\log(\left\vert x-\lambda\right\vert
^{2}+\varepsilon)].
\]
Then $S(\lambda,\varepsilon)$ decreases as $\varepsilon$ decreases, so that
\begin{equation}
s(\lambda)=\lim_{\varepsilon\rightarrow0^{+}}S(\lambda,\varepsilon
)\label{sLim}%
\end{equation}
exists, possibly with the value $-\infty.$ Then $s$ is in $L^1_{\mathrm{loc}}$ and is a subharmonic function. Furthermore, the convergence in (\ref{sLim}) is $L_{\mathrm{loc}}^{1}$ and, therefore,
in the distribution sense.
\end{prop}
The proof will be given after the following definition of the Brown measure; see Section 11.5 of the monograph of Mingo and Speicher \cite{mingo}.
\begin{definition}\label{brown.def}
The \textbf{Brown measure} of an element $x\in\mathcal A$ is defined as the measure $\mu$ computed as
\begin{equation}\label{browndef}
\mu=\frac{1}{4\pi}\Delta s(\lambda),
\end{equation}
where $\Delta$ is the Laplacian in the distribution sense. 
\end{definition}

\begin{remark}
Note that Definition \ref{brown.def} does not, by itself, guarantee that $s$ is the log potential of $\mu$---that is, the convolution of $\mu$ with the function $\log(|\lambda|^2)$. Rather, (\ref{browndef}) only directly tells us that $s$ is the sum of the log potential of $\mu$ and a harmonic function $h$. Nevertheless, for large $\lambda$, we can write $s(\lambda)=\operatorname{tr}[\log(|x-\lambda|^2)]$ without ambiguity, since $\lambda$ will be outside the spectrum of $x$. It is then not hard to see that $s(\lambda)=\log(|\lambda|^2)+o(1)$. The log potential of $\mu$ has the same behavior at infinity, showing that $h$ tends to zero at infinity and must therefore be identically zero. We conclude that $s$ is, actually, the log potential of $\mu$. 
\end{remark}

We now supply the proof of Proposition \ref{limS.prop}.

\begin{proof}
Let $U$ be a nonempty, open, connected subset of $\mathbb{R}^{n}.$ A function
$f:U\rightarrow\lbrack-\infty,\infty)$ is said to be subharmonic if (1) $f$ is
not identically equal to $-\infty,$ (2) $f$ is upper semicontinuous, and (3)
the average of $f$ over a sphere centered at $x\in U$ is greater than or equal
to $f(x),$ whenever the sphere is contained in $U.$ Such a function $f$ is
locally bounded above, because it is upper semicontinuous. Furthermore, $f$ is
in $L_{\mathrm{loc}}^{1}(U)$ and the distributional Laplacian of $f$ is a
non-negative distribution \cite[Theorem 4.1.8]{Hor}. If $f_{n}$ is a weakly
decreasing sequence of subharmonic functions on $U$, the pointwise limit $f$
of $f_{n}$ is easily seen to be subharmonic, provided $f$ is not identically
equal to $-\infty.$ (When computing the averages over spheres, apply monotone
convergence to $f_{1}-f_{n}$.) A smooth function $f:U\rightarrow\mathbb{R}$ is
subharmonic if and only if the Laplacian of $f$ is non-negative \cite[Theorem
2.6.4.2]{Azarin}

We now specialize to the case $n=2$ with $U=\mathbb{C}.$ The function
$S(\lambda,\varepsilon)$ is a smooth function of $\lambda$ for each
$\varepsilon>0$ and the Laplacian of $S$ with respect to $\lambda$ with
$\varepsilon$ fixed is positive \cite[Eq. (11.8)]{mingo}. Now, $S$ can
be computed as
\[
\int_{0}^{\infty}\log(\xi^{2}+\varepsilon)~d\mu_{\left\vert x-\lambda
\right\vert }(\xi),
\]
where $\mu_{\left\vert x-\lambda\right\vert }$ denotes the law (or spectral
distribution)\ of $\left\vert x-\lambda\right\vert .$ After separating the log
function into its positive and negative parts and applying monotone
convergence to the negative part, we see that $s(\lambda)$ can be computed as%
\begin{equation}
s(\lambda)=\int_{0}^{\infty}\log(\xi^{2})~d\mu_{\left\vert x-\lambda
\right\vert }(\xi).\label{sIntegral}%
\end{equation}
Here, the integral of the positive part of the logarithm is finite but the
integral of the negative part can be infinite, meaning that the integral is
well defined but can equal $-\infty.$ If $\lambda$ is outside the spectrum of
$x,$ then $\left\vert x-\lambda\right\vert $ is invertible, so the support of
$\mu_{\left\vert x-\lambda\right\vert }$ of $\left\vert x-\lambda\right\vert $
does not include 0. In that case, the integral in (\ref{sIntegral}) is finite.
We conclude that $s$ is not identically equal to $-\infty$ and is the
decreasing limit of subharmonic functions; therefore $s$ is subharmonic.

We now apply Theorem 4.1.9 in \cite{Hor} to the sequence $f_{n}(\lambda
)=S(\lambda,\varepsilon_{n}),$ for any decreasing sequence $\varepsilon_{n}$
of positive numbers tending to 0. We note that (1) $f_{n}(\lambda)$ is bounded
above by $f_{1}(\lambda),$ and (2) when $\lambda$ is outside the spectrum of
$x,$ the sequence $f_{n}(\lambda)$ is not tending to $-\infty$. Then
\cite[Theorem 4.1.9(a)]{Hor} tells us that $f_{n}$ has a subsequence that
converges in $L_{\mathrm{loc}}^{1}$ to some function $g.$ Then this
subsequence has a sub-subsequence convering pointwise almost everywhere to
$g.$ But the whole sequence $f_{n}$ converges pointwise to $s,$ which means
that $g=s$ almost everywhere. Finally, we apply a standard argument to the
sequence $f_{n}$ in the metric space $L^{1}(K),$ for any compact subset $K$ of
$\mathbb{C}.$ Since every subsequence will have a convergent sub-subsequence
and all the subsequential limits have the same value (namely $s$), the entire
sequence converges to $s$ in $L^{1}(K).$ It is then easily seen that
$S(\lambda,\varepsilon)$ converges to $s$ in $L^{1}(K)$ for every compact set
$K.$ 
\end{proof}

	\section{The case $\tau =s$}\label{sequalstau.sec}
	Recall that we consider the element $xb_{s,\tau}$ where $b_{s,\tau}$ is as in Definition \ref{def b_st} and where $x$ is non-negative and freely independent of $x$. We make the standing assumption that $x$ is not the zero operator. We let $\mu$ denote the law of $x$ as in (\ref{lawdef}). Since $x\neq 0$, the measure $\mu$ will not be a $\delta$-measure at 0. 
	
	We begin by analyzing the case in which $\tau=s$ following the strategy outline in Section \ref{btOutline.sec}.
	\subsection{The PDE for the regularized log potential and its solution}
	Since it is more natural to use $t$ instead of $s$ as the time variable of a PDE, we let $b_t:= b_{t,t}$ be the free multiplicative Brownian motion as defined in Definition \ref{def b_t} and \ref{def b_st}. We then let $x$ be a non-negative operator that is freely independent of $b_t$. We then define $$x_t = xb_t.$$ Consider the functions $S$ and $s_t$ defined by
	\begin{equation}\label{regpot} 
	S(t,\lambda,\varepsilon) = \tr[\log(|x_t-\lambda|^2+\varepsilon)],\quad\varepsilon>0,
	\end{equation}
	and 
	$$ s_t(\lambda) = \lim_{\varepsilon \rightarrow 0^+}S(t,\lambda,\varepsilon) .$$
	Then the density of the Brown measure $W(t,\lambda)$ of $x_t$ can be computed as
	$$ W(t,\lambda) = \frac{1}{4\pi}\Delta_\lambda s_t(\lambda).$$ 
	The function $s_t$ is the log potential of the Brown measure, and we refer to the function $S$ as the ``regularized log potential.''
	
	We use logarithmic polar coordinates $(\rho,\theta)$ defined by $$\lambda=e^\rho e^{i\theta},$$ so that $\rho$ is the logarithm of the usual polar radius $r$. In the case $x=1$, a PDE for $S$ was derived (in rectangular coordinates) in \cite[Theorem 2.7]{dhk}. This derivation applies without change in our situation, as in \cite{hz} in the case of a unitary initial condition. We record the result here.
	\begin{theorem}[Driver--Hall--Kemp]\label{pde_0}
		The function $S$ satisfies the following PDE in logarithmic polar coordinates:
		\begin{equation}\label{pde}
			\frac{\partial S}{\partial t} = \varepsilon\frac{\partial S}{\partial \varepsilon}\left(1+(|\lambda|^2-\varepsilon)\frac{\partial S}{\partial \varepsilon} - \frac{\partial S}{\partial \rho}\right), \quad \lambda = e^\rho e^{i\theta},
		\end{equation}
		with initial condition
		$$S(0,\lambda,\varepsilon) = \tr[\left((x-\lambda)^\ast(x-\lambda)+\varepsilon\right)] = \int_0^{\infty} \log(|\xi-\lambda|^2 +\varepsilon) \,d\mu(\xi) $$
		where $\mu$ is the law of non-negative initial condition $x$.
	\end{theorem}
	
		The PDE (\ref{pde}) is a first-order, nonlinear PDE of Hamilton--Jacobi type. We now analyze the solution using the method of characteristics. See Chapters 3 and 10 of the book of Evans \cite{evans} for more information. See also Section 5 of \cite{dhk} for a concise derivation of the formulas that are most relevant to the current problem. 
		
		We write $\lambda = e^\rho e^{i\theta} = r e^{i\theta}$ and define the Hamiltonian corresponding to (\ref{pde}) by replacing each derivative of $S$ by a ``momentum'' variable, with an overall minus sign: 
	\begin{equation}\label{hamil}
		H(\rho,\theta,\varepsilon,p_\rho,p_\theta,p_\varepsilon) = -\varepsilon p_{\varepsilon}(1+(r^2 - \varepsilon)p_{\varepsilon} - p_{\rho})\\
	\end{equation}
	Now we consider Hamilton's equations for this Hamiltonian:
	\begin{align}
		\frac{d\rho}{dt} &= \frac{\partial H}{\partial p_\rho},\quad \frac{d\theta}{dt} = \frac{\partial H}{\partial p_\theta},\quad \frac{d\varepsilon}{dt} = \frac{\partial H}{\partial p_\varepsilon},\label{Ham1}\\
		\frac{dp_\rho}{dt} &= -\frac{\partial H}{\partial \rho},\quad \frac{dp_\theta}{dt} = -\frac{\partial H}{\partial \theta},\quad \frac{dp_\varepsilon}{dt} = -\frac{\partial H}{\partial \varepsilon}\label{Ham2}
	\end{align}
	Since the right hand side of (\ref{hamil}) is independent of $\theta$ and $p_\theta$, it is obvious that $d\theta/dt = 0 = dp_\theta/dt.$ Thus, $\theta$ and $p_\theta$ are independent of $t$.
	
	To apply the Hamilton--Jacobi method, we take arbitrary initial conditions for the position variables: $$\rho(0) = \rho_0, \quad r(0) = r_0, \quad \varepsilon(0) = \varepsilon_0.$$
	Then the initial conditions for momentum variables, $p_{\rho,0} = p_\rho(0), p_\theta = p_\theta(0)$, and $ p_0 = p_\varepsilon(0) $, are chosen as follows:  
	\begin{align*}
		p_{\rho,0}(\lambda_0, \varepsilon_0) = \frac{\partial S(0,\lambda_0, \varepsilon_0)}{\partial \rho},\\ p_{\theta}(\lambda_0, \varepsilon_0) = \frac{\partial S(0,\lambda_0, \varepsilon_0)}{\partial \theta}, \\ p_{0}(\lambda_0, \varepsilon_0) = \frac{\partial S(0,\lambda_0, \varepsilon_0)}{\partial \varepsilon}.
	\end{align*}
	Recalling that $\mu$ is the law of $x$, we can write the initial momenta explicitly as
	\begin{align}
		p_{\rho,0}(\lambda_0, \varepsilon_0) &= \int_0^{\infty}\frac{2r^2_0 - 2\xi r_0\cos\theta}{|\xi-\lambda_0|^2+\varepsilon_0}\,d\mu(\xi)\label{prho} \\
		p_{\theta,0}(\lambda_0, \varepsilon_0) &= \int_0^{\infty}\frac{2r_0\xi\sin(\theta)}{|\xi-\lambda_0|^2+\varepsilon_0}\,d\mu(\xi)\label{ptheta} \\ 
		p_{0}(\lambda_0, \varepsilon_0)&= \int_0^{\infty}\frac{1}{|\xi-\lambda_0|^2+\varepsilon_0}\,d\mu(\xi). \label{pepsilon}
	\end{align}
	
		The following computations will be useful to us.
	
	\begin{lemma} The Hamiltonian $H$ is a constant of motion for Hamilton's equations and its value at $t=0$ may be computed as follows:
	\begin{equation}\label{h0}H_0 = -\varepsilon_0 p_0p_2,\end{equation}
	where 
	\begin{equation}\label{pkdef}  p_k = \int_0^{\infty}\frac{\xi^k}{|\xi-\lambda_0|^2+\varepsilon_0}\,d\mu(\xi),\quad k=0,2.\end{equation}
	\end{lemma}
	In the $k=0$ case of (\ref{pkdef})), we interpret $\xi^0$ as being identically equal to 1, even at $\xi=0$. Since we assume $x\neq 0$ so that $\mu\neq\delta_0$, neither $p_0$ nor $p_2$ can equal 0.
	\begin{remark}\label{hard.remark}
	If $x=1$, we find that $p_2=p_0$, in which case many of the formulas in the remainder of the paper simplify greatly. (Observe, for example, the simplification in the formula for $\delta$ in Theorem \ref{t*} or the formula for $T$ in Definition \ref{t.def} if $p_2=p_0$.) Meanwhile, if one considers $b_t$ with a unitary rather than non-negative initial condition, as in \cite{hz}, one again has $p_2=p_0$, because the quantity $\xi^2$ in the numerator in (\ref{pkdef}) is really $|\xi|^2$, which would equal $1$ in the unitary case. This observation helps explain why the case of a non-negative initial condition is so much more technically difficult than the case of a unitary initial condition. 
	\end{remark}
	\begin{proof} The Hamiltonian is easily seen to be a constant of motion for any Hamiltonian system. 
		If $\lambda_0=r_0e^{i\theta}$, we compute
		\begin{align*}
			(r_0^2 - \varepsilon_0)p_{\varepsilon,0} - p_{\rho,0} &= \int_0^{\infty}\frac{-r^2_0 + 2\xi r_0\cos\theta - \varepsilon_0}{\xi^2 +r_0^2 -2\xi r_0\cos\theta +\varepsilon_0}\,d\mu(\xi)\\
			&= -1 + p_2.
		\end{align*}
		Then
		$$ H_0 = -\varepsilon_0p_0(1 -1 + p_2) = -\varepsilon_0p_0p_2,$$
		as claimed. 
	\end{proof}

	Now, we are ready to solve the PDE using arguments similar to those in \cite[Section 5]{dhk}.
	\begin{theorem} \label{HJ}
		Assume $\lambda_0 \neq 0$ and $\varepsilon_0 > 0$. Suppose a solution to the system (\ref{Ham1})--(\ref{Ham2}) with initial conditions (\ref{prho})--(\ref{pepsilon}) exists with $\varepsilon(t)> 0$ for $0\leq t \leq T$. Then
		\begin{equation}\label{hjformula}S(t,\lambda(t),\varepsilon(t)) = S(0,\lambda_0,\varepsilon_0) - H_0 t + \log|\lambda(t)| - \log|\lambda_0|,\end{equation}
		for all $ 0\leq t <T$. Moreover, it also satisfies
		\begin{align*}
			\frac{\partial S}{\partial \varepsilon}(t,\lambda(t),\varepsilon(t)) &= p_\varepsilon(t) \\
			\frac{\partial S}{\partial \rho}(t,\lambda(t),\varepsilon(t)) &= p_\rho(t).
		\end{align*}
	\end{theorem}
	\begin{proof}
		We calculate
		\begin{align*}
			p_\rho\frac{d\rho}{dt} +	p_\varepsilon\frac{d\varepsilon}{dt}  &= p_\rho\frac{\partial H}{\partial p_\rho} +	p_\varepsilon\frac{\partial H}{\partial p_\varepsilon} \\
			&= \varepsilon p_\varepsilon - 2H = \varepsilon p_\varepsilon - 2H_0
		\end{align*}
		By Proposition 5.3 in \cite{dhk}, we have 
		$$S(t,\lambda(t),\varepsilon(t)) = S(0,\lambda_0,\varepsilon_0) -  H_0 t + \int_0^t \varepsilon(s)p_\varepsilon(s) \,ds.$$
		Since 
		$$ \frac{d\rho}{dt} = \frac{\partial H}{\partial p_\rho}= \varepsilon p_\varepsilon,$$ 
		we have 
		\begin{equation}\label{loglog} \int_0^t \varepsilon(s)p_\varepsilon(s) \,ds = \log|\lambda(t)| - \log|\lambda_0|.\end{equation}
		Thus, we are done.
	\end{proof}
	\subsection{The lifetime of the solution and its $\varepsilon_0\rightarrow 0$ limit}
	We wish to apply the Hamilton--Jacobi method using the strategy outlined in Section \ref{btOutline.sec}. Thus, we try to choose initial conditions $\lambda_0$ and $\varepsilon_0$ for the Hamiltonian system (\ref{Ham1})--(\ref{Ham2})---with the initial momenta then be determined by (\ref{prho})--(\ref{ptheta})---so that $\lambda(t)$ equals $\lambda$ and $\varepsilon(t)$ equals 0. Our strategy for doing this is to choose $\varepsilon_0=0$ and $\lambda_0=\lambda$, \emph{provided} that the lifetime of the solution remains greater than $t$ as $\varepsilon$ approaches zero. Thus, we need to determine the lifetime and take its $\varepsilon\rightarrow 0$ limit. 
	
	We will eventually want to let $\varepsilon_0$ tend to zero in the Hamilton--Jacobi formula (\ref{hjformula}). We will then want to apply the inverse function theorem to solve for $\lambda_0$ and $\varepsilon_0$ in terms of $\lambda$ and $\varepsilon$. To do this, we need to analyze solutions to the Hamiltonian system (\ref{Ham1})--(\ref{Ham2}) with $\varepsilon_0$ in a neighborhood of 0. Thus, in this section, we allow $\varepsilon_0$ to be slightly negative. We emphasize that even though the Hamiltonian system makes sense for negative values of $\varepsilon,$ the Hamilton--Jacobi formula (\ref{hjformula}) is only applicable when $\varepsilon(s)>0$ for all $s\leq t$, because the regularized log potential $S$ is only defined for $\varepsilon>0$.

	\begin{lemma} \label{pepsilon formula}
	The quantity $$\phi(t)=\varepsilon(t)p_\varepsilon(t)+p_\rho(t)/2$$ is a constant of motion for the Hamiltonian system (\ref{Ham1})--(\ref{Ham2}). Then if we let $C=2\phi(0)-1$, we have
		\begin{equation}\label{ep2}\varepsilon(t)p_{\varepsilon}(t)^2 = \varepsilon_0p_{\varepsilon,0}^2e^{-Ct}\end{equation}
		for all $t$. The constant $C$ may be computed as 
		\[ C= p_0(r_0^2+\varepsilon_0 )-p_2,
		\]
		where $p_0$ and $p_2$ are defined by (\ref{pkdef}).
	\end{lemma}
	\begin{proof} It is an easy computation to show, using (\ref{Ham1})--(\ref{Ham2}), that $d\phi/dt=0$ and that $$\frac{d}{dt}(\varepsilon p_{\varepsilon}^2)=-\varepsilon p_{\varepsilon}^2(2\phi-1).$$
			Then since $\phi$ is a constant of motion, we obtain the claimed formula for $\varepsilon p_{\varepsilon}^2$. Meanwhile, we can compute that
			\begin{equation*}
			2\phi(0)-1=\int_0^{\infty} \frac{2\varepsilon_0+(2r^2_0 - 2\xi r_0\cos(\theta))-(\xi^2 +r_0^2 -2\xi r_0\cos\theta +\varepsilon_0)}{\xi^2 +r_0^2 -2\xi r_0\cos\theta +\varepsilon_0}\,d\mu(\xi),
		\end{equation*}
		which simplifies to the claimed expression for $C$.
		\end{proof}

	We are now ready to compute the blow-up time of the solutions to the Hamiltonian system. We recall the definition in (\ref{pkdef}) of the quantities $p_0$ and $p_2$.
	
	\begin{theorem}\label{t*} 
	Take $\lambda_0\neq 0$ with $\lambda_0$ outside $\operatorname{supp}(\mu)$. 
		As long as $\varepsilon_0$ is not too negative, the blow-up time $t_\ast$ of the system (\ref{Ham1})--(\ref{Ham2}) is
		$$t_{\ast} = \frac{1}{r_0\sqrt{p_2p_0} \sqrt{\delta^2-4}}\log\left(\frac{\delta + \sqrt{\delta^2-4}}{\delta - \sqrt{\delta^2-4}}\right),$$
		where
		\begin{equation}\label{deltadef}
			\delta = \frac{p_0r^2_0 + p_2 + p_0\varepsilon_0}{r_0\sqrt{p_2p_0}}.
		\end{equation}
		If $\varepsilon_0$ is positive, $\varepsilon(t)$ will remain positive for all $t<t_\ast$. 
	\end{theorem}
	The precise assumptions on $\varepsilon_0$ are given in the first paragraph of the proof.
	\begin{proof}
	If $\lambda_0$ is outside $\operatorname{supp}(\mu)$ and $\varepsilon_0$ is not too negative, the initial momenta in (\ref{prho}), (\ref{ptheta}), and (\ref{pepsilon}) will be well defined and the quantities $p_k$ in (\ref{pkdef}) will be well defined and positive. Specifically, we need $\varepsilon_0> -d(\lambda_0,\operatorname{supp}(\mu))^2$. In what follows, the statement ``$\varepsilon_0$ is slightly negative'' will mean that $\varepsilon_0<0$ is chosen to be greater than $-d(\lambda_0,\operatorname{supp}(\mu))^2$ and so that the quantity $\delta$ in (\ref{deltadef}) remains positive.

Recall that
\[
\frac{dp_{\varepsilon}}{dt}=-\frac{\partial H}{\partial\varepsilon}=-\frac
{H}{\varepsilon}-\varepsilon p_{\varepsilon}^{2}.
\]
Using Lemma \ref{pepsilon formula}, we can express $dp_{\varepsilon}/dt$ in a
form that involves only $p_{\varepsilon}$ and constants of motion:
\[
\frac{dp_{\varepsilon}}{dt}=-\frac{H}{\varepsilon_{0}p_{0}^{2}}p_{\varepsilon
}^{2}e^{Ct}-\varepsilon_{0}p_{0}^{2}e^{-Ct}.
\]
Let $B=\varepsilon_{0}p_{0}^{2}$ and $y(t)=-\frac{H}{B}p_{\varepsilon}%
e^{Ct}+\frac{C}{2}$. Since a Hamiltonian $H$ is a constant of motion, we
compute that, on the one hand
\[
\frac{dy}{dt}=\frac{H^{2}}{B^{2}}p_{\varepsilon}^{2}e^{2Ct}-\frac{HC}%
{B}p_{\varepsilon}e^{Ct}+H,
\]
but on the other hand,
\[
y^{2}=\frac{H^{2}}{B^{2}}p_{\varepsilon}^{2}e^{2Ct}-\frac{HC}{B}%
p_{\varepsilon}e^{Ct}+\frac{C^{2}}{4}.
\]

We therefore find that
\begin{equation}
y^{2}-\frac{dy}{dt}=a^{2},\label{sepODE}%
\end{equation}
where
\[
a^{2}=\frac{C^{2}}{4}-H.
\]
Now, (\ref{sepODE}) is a separable differential equation which can be solved
as in Lemma 5.8 in \cite{dhk}, as
\begin{equation}
y(t)=\frac{y_{0}\cosh(at)-a\sinh(at)}{\cosh(at)-y_{0}\frac{\sinh(at)}{a}%
}.\label{yoft1}%
\end{equation}
We note that the right-hand side of (\ref{yoft1}) is an even function of $a,$
so that either of the square roots of $a^{2}$ may be used. Now, $y(t)$ (and
therefore also $p_{\varepsilon}(t)$) will blow up when the denominator is
zero, i.e., at $t=t_{\ast}$ where
\begin{equation}
t_{\ast}=\frac{1}{2a}\log\left(  \frac{1+a/y_{0}}{1-a/y_{0}}\right)
.\label{tstarFirst}%
\end{equation}

We now compute the value of $a^{2}$ as
\begin{align}
a^{2} &  =\frac{C^{2}}{4}-H\nonumber\\
&  =\frac{p_{0}^{2}(r_{0}^{2}+\varepsilon_{0}-\frac{p_{2}}{p_{0}})^{2}}%
{4}+\varepsilon_{0}p_{0}p_{2}.\label{a2value1}%
\end{align}
We further simplify this result as%
\begin{align}
a^{2} &  =\frac{p_{0}^{2}}{4}\left(  \left(  r_{0}^{2}+\varepsilon_{0}%
+\frac{p_{2}}{p_{0}}\right)  ^{2}-4r_{0}^{2}\frac{p_{2}}{p_{0}}\right)
\label{a2value2}\\
&  =\frac{p_{0}p_{2}r_{0}^{2}}{4}\left(  \delta^{2}-4\right)  .\nonumber
\end{align}
We then compute
\begin{align*}
y_{0} &  =p_{2}+\frac{p_{0}r_{0}^{2}+p_{0}\varepsilon_{0}-p_{2}}{2}\\
&  =\frac{p_{0}r_{0}^{2}+p_{0}\varepsilon_{0}+p_{2}}{2}\\
&  =\frac{r_{0}\sqrt{p_{2}p_{0}}}{2}\delta.
\end{align*}

We can then find the value of $p_{\varepsilon}(t)$ from the value of
$y(t),$ with the result that%
\begin{equation}
p_{\varepsilon}(t)=p_{0}e^{-Ct}\frac{\sqrt{\delta^{2}-4}\cosh(at)+(2r_{0}%
\sqrt{p_{0}/p_{2}}-\delta)\sinh(at)}{\sqrt{\delta^{2}-4}\cosh(at)-\delta
\sinh(at)}.\label{pEpsFormula}%
\end{equation}
This expression is very similar to the one in \cite[Eq. (5.45)]{dhk}, with the
only difference being the factor of $\sqrt{p_{0}/p_{2}}$ in the second term in
the numerator. (This factor does not appear in \cite{dhk} because $p_{2}%
=p_{0}$ when $x=1.$) We now claim that if $\varepsilon_{0}$ is either
non-negative or slightly negative, the solution to the whole Hamiltonian
system (\ref{Ham1})--(\ref{Ham2}) will exist up to the time $t_{\ast}$ when
the denominator on the right-hand side of (\ref{pEpsFormula}) becomes zero. A
key step is to solve for $\varepsilon(t)$ in (\ref{ep2}) as
\begin{equation}
\varepsilon(t)=\frac{1}{p_{\varepsilon}(t)^{2}}\varepsilon_{0}p_{0}^{2}%
e^{-Ct}.\label{epsilonoft}%
\end{equation}
This formula is meaningful as long as $p_{\varepsilon}(t)$ remains nonzero. 

We now claim that $p_{\varepsilon}(t)$ remains positive until it blows up,
even if $\varepsilon_{0}$ is slightly negative. We consider first the case
$\varepsilon_{0}\geq0.$ In that case, by (\ref{a2value1}) and (\ref{a2value2}), $a^{2}\geq0$
and $\delta\geq 4$. We then claim that the coefficient of
$\cosh(at)$ in the numerator of the right-hand side of (\ref{pEpsFormula}) is
at least as big as the absolute value of the coefficient of $\sinh(at).$ Thus,
the numerator will be positive for all $t>0.$ To verify this claim, we
compute that
\[
\delta^{2}-4-(2r_{0}\sqrt{p_{0}/p_{2}}-\delta)^{2}=\frac{4p_{0}\varepsilon
_{0}}{p_{2}}\geq0.
\]

We consider next the case in which $\varepsilon_{0}$ is slightly negative.
From (\ref{a2value1}), we see that typically, $a^{2}$ remains positive even
when $\varepsilon_{0}$ becomes slightly negative, in which case, the argument is
as in the case $\epsilon_0\geq 0$. But if $r_{0}^{2}%
=p_{2}/p_{0}$ at $\varepsilon_{0}=0,$ the value of $a^{2}$---and thus the
value of $\delta^{2}-4$---will become negative when $\varepsilon_{0}$ is
slightly negative. In that case, we write $a=i\alpha,$ where we can choose
$\alpha>0.$ Then the expression in (\ref{pEpsFormula}) becomes%
\begin{equation}
p_{\varepsilon}(t)=p_{0}e^{-Ct}\frac{\sqrt{4-\delta^{2}}\cos(\alpha t)-(\delta-2r_{0}
\sqrt{p_{0}/p_{2}})\sin(\alpha t)}{\sqrt{4-\delta^{2}}\cos
(\alpha t)-\delta\sin(\alpha t)}.\label{pEpsNeg}%
\end{equation}

The numerator on the right-hand side of (\ref{pEpsNeg}) becomes zero at the
time
\begin{equation*}
t_1=\frac{1}{\alpha}\tan^{-1}\left(  \frac{\sqrt{4-\delta^{2}}%
}{\delta-2r_{0}\sqrt{p_{0}/p_{2}}}\right)  
\end{equation*}
while denominator becomes zero at%
\begin{equation*}
t_2=\frac{1}{\alpha}\tan^{-1}\left(  \frac{\sqrt{4-\delta^{2}}%
}{\delta}\right)  .
\end{equation*}
Let $\theta_1$ and $\theta_2$ denote the arguments of 
the inverse tangents in the formulas for $t_1$ and $t_2$, respectively. 
Then $\theta_2$ is positive, so that the
value of the inverse tangent is in $(0, \pi/2).$ Now, $\theta_1$  could be positive and bigger than
$\theta_2$, in which case $t_1$ is bigger
than $t_2$ Alternatively, $\theta_1$ could be negative, in which case, to get a positive value of
$t_1,$ the value of the inverse tangent must be bigger than
$\pi/2.$ Either way, $t_1$ is greater than $t_2,$
showing that $p_{\varepsilon}(t)$ remains positive until it blows up. 

Since $p_{\varepsilon}(t)$ remains positive until it blows up, the formula for
$\varepsilon(t)$ in (\ref{epsilonoft}) remains nonsingular up to time
$t_{\ast}.$ We can then follow the proof of \cite[Proposition 5.11]{dhk} to
construct a solution to the system (\ref{Ham1})--(\ref{Ham2}) up to time
$t_{\ast}$.
		
		Finally, by (\ref{epsilonoft}), if $\varepsilon_0>0$, then $\varepsilon(t)$ will remain positive for $t<t_\ast$.
	\end{proof}
	
		We now supply the proof of Lemma \ref{smalleps.lem}, stating that in the limit as $\varepsilon_0\rightarrow 0$, we obtain $\varepsilon(t)\equiv 0$ and $\lambda(t)\equiv\lambda_0$.\label{limitproof}
		
		\begin{proof}[Proof of Lemma \ref{smalleps.lem}]
		Since $\lambda_0$ is assumed to be outside $\supp(\mu)$, we may take $\varepsilon\rightarrow 0$ in (\ref{epsilonoft}) and $\varepsilon_0$ and $p_0$ will remain finite. Furthermore, as discussed in the proof of Theorem \ref{t*}, as long as $\varepsilon_0$ is at most slightly negative, $p_\varepsilon(t)$ will remain positive for as long as the solution to the whole system exists. Thus, $\varepsilon(t)$ becomes identically zero in the limit, until the solution of the system ceases to exist. Meanwhile, when $\varepsilon_0$ approaches 0 (so that $\varepsilon(t)$ also approaches 0), we can see from (\ref{loglog}) that $\vert\lambda(t)\vert$ approaches $\vert\lambda_0\vert$. Since, the argument of $\lambda(t)$ is a constant of motion, the proof is complete.
		\end{proof}

	We now study the behavior of the lifetime $t_\ast$ in the limit as $\varepsilon$ tends to zero.
	\begin{definition}\label{t.def}Define the function $T : \mathbb{C}\backslash\supp(\mu) \rightarrow [0, \infty)$ by 
	\begin{align*}
		T(\lambda_0) = \begin{cases}
			\frac{\log(\tilde{p}_2) - \log(\tilde{p}_0r_0^2)}{ \tilde{p}_2 - \tilde{p}_0r_0^2 } &\tilde{p}_0r_0^2 \neq \tilde{p}_2 \\
			\frac{1}{\tilde{p}_2} &\tilde{p}_0r_0^2 = \tilde{p}_2
		\end{cases},
	\end{align*}
	where $r_0=\left\vert\lambda_0\right\vert$ and $\tilde{p}_k$ denotes the value of the quantity $p_k$ in (\ref{pkdef}) at $\varepsilon_0=0$: 
	\begin{equation}\label{pktildedef}  \tilde{p}_k = \int_0^{\infty}\frac{\xi^k}{|\xi-\lambda_0|^2}\,d\mu(\xi),\quad k=0,2.\end{equation}
	\end{definition}
	
	Note that $T$ has a removable singularity at $\tilde{p}_0r_0^2 = \tilde{p}_2$, so that $T$ is an analytic function of the positive quantities $\tilde{p}_0$, $\tilde{p}_2$, and $r_0$. From the definition of $T$ and (\ref{pktildedef}), we can easily see that $$T(\overline{\lambda_0})=T(\lambda_0).$$
	\begin{prop}\label{zerooutside}
		For all $\lambda_0$ outside $\supp(\mu)$, we have
		$$\lim_{\varepsilon_0 \rightarrow 0} t_\ast(\lambda_0,\varepsilon_0) = T(\lambda_0).$$
		Moreover, when $\lambda_0 =0$ is outside $\supp(\mu)$, we have
		$$\lim_{\varepsilon_0 \rightarrow 0} t_\ast(0,\varepsilon_0) = T(0) = \infty.$$
	\end{prop}
	\begin{proof}
	Recall the definition of $\delta$ in Theorem \ref{t*}. The key point is that when $\varepsilon_0=0$, we have 
	\[
	\delta^2-4=\frac{(\tilde{p}_0 r_0-\tilde{p}_2)^2}{\tilde{p}_0 \tilde{p}_2r_0^2}.
	\]
	We then note that the quantity
	\[
	\frac{1}{\sqrt{\delta^2-4}}\log\left(\frac{\delta + \sqrt{\delta^2-4}}{\delta - \sqrt{\delta^2-4}}\right)
	\]
	has the same value no matter which square root of $\delta^2-4$ we use. It is therefore harmless to choose 
	\[
	\sqrt{\delta^2-4}=\frac{\tilde{p}_0 r_0-\tilde{p}_2}{r_0 \sqrt{\tilde{p}_0 \tilde{p}_2}}.
	\]
	In that case, we find that
	
		\[ \lim_{\varepsilon_0 \rightarrow 0}\frac{\delta + \sqrt{\delta^2-4}}{\delta - \sqrt{\delta^2-4}} = \frac{\tilde{p}_0r_0^2}{\tilde{p}_2}\]
	and
	\[ 
	\lim_{\varepsilon_0 \rightarrow 0}r_0\sqrt{p_0 p_2}\sqrt{\delta^2 -4} = \tilde{p}_0r_0^2 -\tilde{p}_2
	\]
	and the claimed formula holds, provided $\tilde{p}_0 r_0^2\neq \tilde{p}_2$.
	
	In the case $\tilde{p}_0 r_0^2= \tilde{p}_2$, we have $\lim_{\varepsilon_0 \rightarrow 0} \delta = 2$ and, so that
		$$ \lim_{\varepsilon \rightarrow 0} \frac{1}{ \sqrt{\delta^2-4}}\log\left(\frac{\delta + \sqrt{\delta^2-4}}{\delta - \sqrt{\delta^2-4}}\right) = 1$$
		and $$\lim_{\varepsilon_0 \rightarrow 0} t_\ast  = \frac{1}{\tilde{p}_2}$$
		as claimed.
		
		Finally, when $\lambda_0 = 0 \notin \supp(\mu)$, we have that $\tilde{p}_0r_0^2 = 0< 1 =\tilde{p}_2.$ Then 
		\[ \lim_{\varepsilon_0 \rightarrow 0}\frac{\delta + \sqrt{\delta^2-4}}{\delta - \sqrt{\delta^2-4}} = \infty \quad \text{and } \lim_{\varepsilon_0 \rightarrow 0}r_0p_0\frac{\sqrt{\tilde{p}_2}}{\sqrt{\tilde{p}_0}}\sqrt{\delta^2 -4} = 1 .\]
		Thus,
		\[\lim_{\varepsilon \rightarrow 0} t_\ast(0,\varepsilon_0) = \infty = T(0),\]
		completing the proof.		
	\end{proof} 
	Next, we show that, for nonzero $\lambda$, $\lim_{\theta \rightarrow 0^+} T(\lambda) = 0$ $\mu$-almost everywhere. First, we state a result Zhong \cite[Lemma 4.4]{pz}, as follows.
	\begin{lemma}[Zhong] \label{pz4.4}
		Let $\mu$ be a nonzero, finite Borel measure on $\mathbb{C}$. Define $I : \mathbb{C} \rightarrow (0, \infty]$ by
		$$ I(\lambda) = \int_\mathbb{C} \frac{1}{|z-\lambda|^2} \,d\mu(z).$$
		Then $I(\lambda)$ is infinite almost everywhere relative to $\mu$.
	\end{lemma}
	\begin{lemma}\label{log.lem}
	Let $x_n$ and $y_n$ be sequences of positive real numbers such that $x_n\rightarrow\infty$ and such that $y_n\geq a$ for some constant $a>0$. Then
		\[\lim_{n\rightarrow \infty} \frac{\log(x_n) -\log(y_n)}{x_n-y_n} = 0. \]
	Similarly, if $x_n\leq b<\infty$ and $y_n\rightarrow 0$, then
	\[\lim_{n\rightarrow \infty} \frac{\log(x_n) -\log(y_n)}{x_n-y_n} = +\infty. \]
	\end{lemma}
	\begin{proof}
		We first claim that the function $\frac{\log(x) -\log(y)}{x-y}$ is decreasing in $x$ with $y$ fixed---and therefore, by symmetry, decreasing in $y$ with $x$ fixed. Taking the derivative with respect to $x$, we get
		\begin{align*}
			\frac{d }{dx}\left(\frac{\log(x) -\log(y)}{x-y}\right) &=\frac{(x-y)\frac{1}{x} + \log(\frac{y}{x})}{(x-y)^2} \\
			&\leq \frac{(x-y)\frac{1}{x}  +\frac{y}{x}-1}{(x-y)^2} \\
			&=0,
		\end{align*}
		where we have used the elementary inequality $\log(x) \leq x-1$.
		Using this result, we find, in the first case, that
		\[
		0\leq	\frac{\log(x_n) -\log(y_n)}{x_n-y_n} \leq\frac{\log(x_n) -\log(a)}{x_n-a}\rightarrow 0
		\]
		and, in the second case, that 
		\[
		\frac{\log(x_n) -\log(y_n)}{x_n-y_n}\geq \frac{\log(b) -\log(y_n)}{b-y_n}\rightarrow +\infty,
		\]
		as claimed.
	\end{proof}
	We remind the reader of our standing assumption that $x\neq 0$ so that $\mu\neq\delta_0$.
	\begin{prop}\label{valuet}
		The function $T$ in Definition (\ref{t.def}) satisfies:
		\begin{enumerate}
		\item $\lim_{\theta_0 \rightarrow 0} T(r_0e^{i\theta_0})$ exists for every nonzero $r_0$,
		\item $\lim_{\theta_0 \rightarrow 0} T(r_0e^{i\theta_0}) = 0$ for $\mu$-almost every nonzero $r_0$, and
		\item $\lim_{\lambda_0\rightarrow\infty}T(\lambda_0)=+\infty.$
		\end{enumerate}
	\end{prop}
	Point (2) says that $T(r_0e^{i\theta_0})$ approaches zero as $\theta_0$ approaches zero, for ``most'' (but not necessarily all) nonzero $r_0$ inside the support of $\mu$. The proposition allows us to extend the definition of $T$ from $\mathbb C\setminus\supp(\mu)$ to all of $\mathbb C\setminus \{0\}$. This extension, however, may not be continuous. 
	\begin{proof}
		We write the quantities $\tilde{p}_k$ in (\ref{pkdef}) in polar coordinates as
		\[
		\int_0^{\infty} \frac{\xi^k}{r_0^2+\xi^2-2\xi r_0\cos(\theta_0)}\,d\mu(\xi).
		\]
		Since $\tilde{p}_k$ is an even function of $\theta_0$, so is $T(r_0 e^{i\theta_0})$. We therefore consider only the limit at $\theta_0$ approaches 0 from above. We then note that $\theta_0$ decreases toward zero, the value of $\tilde{p}_k$ increases. Thus, by the monotone convergence theorem,
		\begin{equation}\label{pklim}
		\lim_{\theta_0\rightarrow 0}\tilde{p}_k=\int_0^{\infty}\frac{\xi^k}{(\xi-r_0)^2}\,d\mu(\xi).
		\end{equation}
		Point (1) of the proposition is then clear in the case that $\tilde{p}_0$ and $\tilde{p}_2$ are finite at $\theta_0=0$. 
		
		We then consider the case when at least one of $\tilde{p}_0$ and $\tilde{p}_2$ is infinite at $\theta_0=0$, in which case, $\tilde{p}_0$ must be infinite. Now, since $\mu\neq\delta_0$, the value of $\tilde{p}_2$ at $\theta_0 = 0$ is not zero. Furthermore, the value of $\tilde{p}_2$ at any $\theta_0$ is bounded below by its (positive) value at $\theta_0=0$. Thus, by (\ref{pklim}) and Lemma \ref{log.lem}, the limit of $T(r_0 e^{i\theta_0})$ exists and is zero.
		We conclude that the limit in Point (1) of the proposition exists for all nonzero $r_0$ and that this limit is zero whenever the value of $\tilde{p}_0$ at $\theta_0=0$ is infinite. But by Lemma \ref{pz4.4}, $\tilde{p}_0=\infty$ (at $\theta_0=0$) for $\mu$-almost every nonzero value of $r_0$.
		
		Finally, if $\lambda_0\rightarrow\infty$, then the quantities $\tilde p_0$ and $\tilde p_2$ in (\ref{pktildedef}) tend to zero, so that by the second part of Lemma \ref{log.lem}, $T(\lambda_0)$ tends to $+\infty$.
	\end{proof}
	\begin{remark}\label{allsupport.rem}
	In some cases, Point 2 of Proposition \ref{valuet} will hold for \emph{every} $r_0$ in the support of $\mu$. By examining the preceding proof, we see that this result will hold if
	\begin{equation}\label{p0infinity}
	\tilde p_0:=\int_0^\infty\frac{1}{(\xi-r_0)^2}\,d\mu(\xi)=\infty,\quad\forall r_0\in\supp(\mu).
	\end{equation}
	The condition (\ref{p0infinity}) will hold if, for example, $\mu$ is a finite sum of measures each of which is either a point mass or is supported on a closed interval with a Lebesgue density that is bounded away from zero on that interval. 
	\end{remark}
	
	\subsection{The domain and its properties}
	In the previous section, we obtain a lifetime $t_\ast$ of the solution to (\ref{Ham1})--(\ref{Ham2}) and its limit $T$ as $\varepsilon_0
 \rightarrow 0$, where $T$ is given outside the support of $\mu$ by Definition \ref{t.def}. Also by Proposition \ref{valuet}, we can extend the domain of $T(\lambda_0)$ to every nonzero $\lambda_0$ by letting $\theta_0$ tend to $0$. 
	
	\begin{definition}\label{sigma.def}
		For all $t>0$, define a domain $\Sigma_t$ by
		\[\Sigma_t = \{\lambda_0 \neq 0 : T(\lambda_0) < t\}.\]
	\end{definition} 
	By the last part of Proposition \ref{valuet}, $\Sigma_t$ is bounded for every $t>0$.

	\begin{corollary}\label{suppoutside}
		For all $t>0$, we have that $\supp(\mu)\setminus\{0\}$ is contained in  $\overline{\Sigma}_t$. In particular, if $0$ is in the support of $\mu$ but outside $\overline{\Sigma}_t$, then $0$ is an isolated point mass of $\mu$.
	\end{corollary}  
	\begin{proof}
	By Point (2) of Proposition \ref{valuet}, the closed set $\overline\Sigma_t$ contains $\mu$-almost every nonzero real number.
	\end{proof}
	
	\begin{figure}
		\centering
		\includegraphics[width=0.5\linewidth]{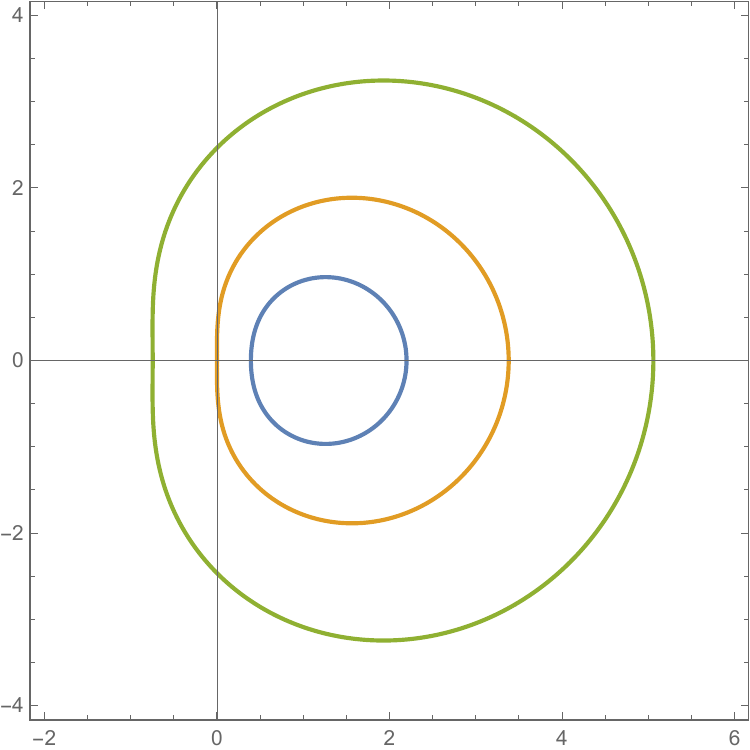}
		\caption{The domain $\compconj{\Sigma}_t$ with $\mu = \frac{1}{2}\delta_0 +\frac{1}{2}\delta_1$, for $t=1$ (blue), $t=2$ (orange), and $t=3$ (green).}
	\end{figure}
	
		\begin{figure}
		\centering
		\includegraphics[width=0.5\linewidth]{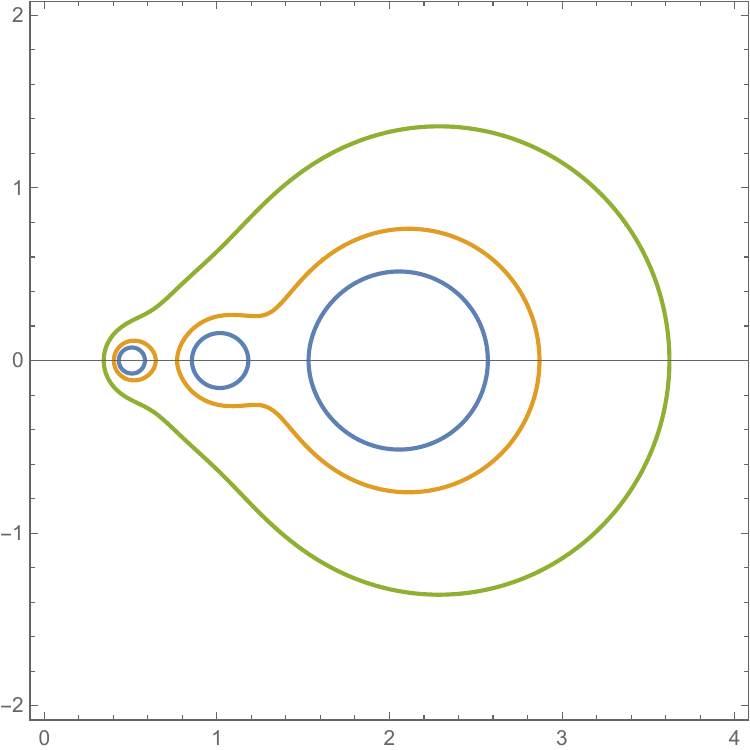}
		\caption{The domain $\compconj{\Sigma}_t$ with $\mu =\frac{1}{5}\delta_{\frac{1}{2}} +\frac{1}{5}\delta_1 + \frac{3}{5}\delta_{2}$, for $t =\frac{1}{10}$ (blue), $t=\frac{1}{5}$ (orange), and $t=\frac{1}{2}$ (green).}
		\label{fig:3pointstest}
	\end{figure}
	The analogous domain to $\Sigma_t$ in \cite{dhk} and \cite{hz} can be described as the set of points $re^{i\theta}$ with $1/r_t(\theta)<r<r_t(\theta)$, for a certain function $r_t$. We now give a similar description of our domain, but with the roles of $r$ and $\theta$ reversed. 
	\begin{definition}
		Let $\pi^+$ be any number larger than $\pi$. Define $\theta_t : (0, \infty) \rightarrow [0, \pi] \cup\{ \pi^+\}$ by 
		$$\theta_t(r_0)= \inf\{\theta_0 \in (0, \pi] : T(r_0e^{i\theta_0}) \geq t\},$$
		if the set is nonempty. Otherwise, $\theta_t(r_0) = \pi^+$.
	\end{definition}

	\begin{prop}\label{equidomain}
		The domain $\Sigma_t$ in Definition \ref{sigma.def} can be characterized as
		\[\Sigma_t =  \{r_0e^{i\theta_0} : r_0 \neq 0 \text{ and } |\theta_0|<\theta_t(r_0)\}.\]	
	\end{prop}

		Figure \ref{fig:domainex} illustrates the meaning of three cases $\theta_t<\pi$, $\theta_t=\pi$, and $\theta_t=\pi^+$. When $\theta_t(r_0)=\pi^+$, the entire circle of radius $r_0$ (centered at the origin) is contained in $\Sigma_t$. When $\theta_t(r_0)=\pi$, the entire circle of radius $r_0$, \emph{except} the point on the negative real axis, is contained in $\Sigma_t$.
		
		The following result is the key to proving Proposition \ref{equidomain}.
	\begin{lemma}\label{monotoneT}
		For each $\lambda_0$ with $\Im(\lambda_0) \neq 0$, the sign of $\partial T(\lambda_0)/\partial\theta$ is the same as the sign of $\Im(\lambda_0)$. 
	\end{lemma}
	The lemma says that $T$ increases as we move away from the positive $x$-axis in the radial direction.
	
	\begin{figure}
		\centering
		\includegraphics[width=0.7\linewidth]{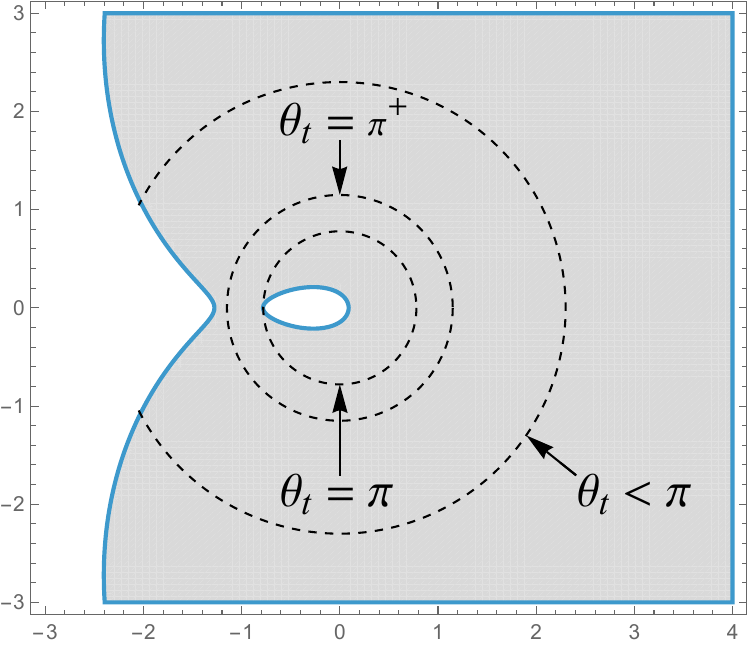}
		\caption{A portion of the domain $\Sigma_t$ with $\mu=\delta_1$ and $t=4.02$}
		\label{fig:domainex}
	\end{figure}

	\begin{proof} 

	We use the subscript notation $f_\theta$ for the partial derivative of a function $f$ in the angular direction. Since $T(\overline{\lambda_0})=T(\lambda_0)$, we see that $T(r_0 e^{i\theta})$ is an even function of $\theta$. It therefore suffices to show that $T_\theta$ is positive when $\Im(\lambda_0)$ is positive. 
	Let $$R = \frac{\tilde{p}_0r_0^2}{\tilde{p}_2}.$$ Note that $T$ can be computed as
		$$ T(r_0e^{i\theta}) = \frac{1}{\tilde{p}_2(R-1)}\log(R).$$
		
		We first make a preliminary calculation:
		\begin{equation*}
		T_\theta = \left(\frac{1}{\tilde{p}_2 (R-1)}\right)_\theta \log(R)+\frac{1}{\tilde{p}_2(R-1)}\frac{R_\theta}{R}.
		\end{equation*}
		We then use the inequality
		\begin{equation}\label{logineq}1-1/x\leq \log x \leq x-1,\end{equation}
		which may be proved by writing $\log x$ as the integral of $1/y$ from 1 to $x$ and then bounding $1/y$ between 1 and $1/x$ (for $x<1$) and between $1/x$ and 1 (for $x>1$). 
		
		In the case that $\left(\frac{1}{\tilde{p}_2 (R-1)}\right)_\theta$ is positive, we use the first inequality in (\ref{logineq}) to give	
		\begin{equation*}
		T_\theta\geq \left(\frac{1}{\tilde{p}_2 (R-1)}\right)_\theta (1-1/R)+\frac{1}{\tilde{p}_2(R-1)}\frac{R_\theta}{R}.
		\end{equation*}
		Simplifying this result gives
		\begin{align*}
			T_\theta \geq   -\frac{(\tilde{p}_2)_\theta}{\tilde{p}_2^2R}.
		\end{align*}
		But
		\begin{align*}
			(\tilde{p}_2)_\theta = -\Im(\lambda_0)\int_0^\infty \frac{\xi^3}{(\xi^2 -2\xi r_0\cos \theta + r_0^2)^2} \,d\mu(\xi).
		\end{align*}
		Thus, $T_\theta$ is positive when $\Im(\lambda_0)$ is positive and we are done in this case.
		
		In the case that $\left(\frac{1}{\tilde{p}_2 (R-1)}\right)_\theta$ is negative, we use the second inequality in (\ref{logineq}) to give
		\begin{equation*}
		T_\theta\geq \left(\frac{1}{\tilde{p}_2 (R-1)}\right)_\theta (R-1)+\frac{1}{\tilde{p}_2(R-1)}\frac{R_\theta}{R}.
		\end{equation*}
		Simplifying this result gives
		\begin{align*}
		T_\theta &\geq - \frac{[(\tilde{p}_2)_\theta R+\tilde{p}_2R_\theta]}{\tilde{p}_2^2 R} \\
			&= -\frac{(\tilde{p}_2R)_\theta}{\tilde{p}_2^2R}.
		\end{align*}
		But since $\tilde{p}_2R = \tilde{p}_0r_0^2$, we compute that
		\begin{align*}
			[\tilde{p}_2R]_\theta = -\Im(\lambda_0)\int_0^\infty \frac{r_0^2\xi}{(\xi^2 -2\xi r_0\cos \theta + r_0^2)^2} \,d\mu(\xi).
		\end{align*} 
		Thus, $T_\theta$ is positive when $\Im(\lambda_0)$ is positive and we are done in this case.
	\end{proof}

	\begin{proof}[Proof of Proposition $\ref{equidomain}$.]
		Suppose that $\lambda_0 =r_0e^{i\theta_0}$ with $\theta_0 \neq 0$. Then by the monotonicity of $T(r_0e^{i\theta_0})$ with respect to $\theta_0$ (Lemma \ref{monotoneT}), $|\theta_0| < \theta_t(r_0)$ if and only if $T(\lambda_0) <t$. \qedhere
	\end{proof}

		\begin{prop}\label{openness}
		For $t>0$,	$\Sigma_t$ is open.
	\end{prop}
	To show that $\Sigma_t$ is open, we first show that $\theta_t$ is continuous in the following sense.
	\begin{lemma}\label{conttheta}
		Let $r_0 >0$. We have 
		\begin{itemize}
			\item[$(1)$] If $\theta_t(r_0) = \pi^+$, there exists a neighborhood $B$ of $r_0$ such that $\theta_t(s) = \pi^+$ for all $s \in B$.
			\item[$(2)$] If $\theta_t(r_0) = \pi$, then for all $\varepsilon >0$, there exists a neighborhood $B$ of $r_0$ such that $\theta_t(s) \in (\pi-\varepsilon,\pi]$ or $\theta_t(s) = \pi^+$ for all $s \in B$.
			\item[$(3)$] If $\theta_t(r_0) < \pi$, then for all $\varepsilon>0$,  there exists a neighborhood $B$ of $r_0$ such that $|\theta_t(r_0) - \theta_t(s)| <\varepsilon$ for all $s \in B_3$.
		\end{itemize}
		
	\end{lemma}
	\begin{proof} The function $T$ in Definition \ref{t.def} is continuous outside $\supp(\mu)$ and, in particular, outside $[0, \infty)$. 
	
	First, assume $\theta_t(r_0) = \pi^+$. Note that by monotonicity of $T(r_0e^{i\theta_0})$ with respect to $\theta_0$ (Lemma $\ref{monotoneT}$), $\theta_t(r_0) = \pi^+$ if and only if  $T(r_0e^{i\pi}) <t$. By the continuity of $T$, there exists a neighborhood $B$ of $r_0$ such that $T(s e^{i\pi}) <t$ for all $s \in B_1$. Hence $\theta_t(s) = \pi^+$ for all $s \in B$.
		
		Next, assume $\theta_t(r_0) = \pi$. Let $\varepsilon>0$, where it is harmless to assume $\varepsilon<\pi$. Then, by the definition of $\theta_t$, we must have $T(r_0 e^{i(\pi-\varepsilon)})<t$. Thus, by the continuity of $T$, there is some neighborhood $B$ of $r_0$ such that $T(s e^{i(\pi-\varepsilon)})<t$ for all $s\in B$. Finally, by the monotonicity of $T$, we must have $\theta_t(s)>\pi-\varepsilon$ for $s\in B$, meaning that either $\theta_t(s)\in (\pi-\varepsilon,\pi]$ or $\theta_t(s)=\pi^+$.
		
		Finally, assume $\theta_t(r_0) < \pi.$ Let $\varepsilon >0$, where it is harmless to assume that $\theta_t(r_0)+\varepsilon<\pi$ and (in the case $\theta_t(r_0)>0$) that $\theta_t(r_0)-\varepsilon>0$. Then by monotonicity of $T$, we have $T(r_0e^{i(\theta_t(r_0)+\varepsilon)}) > t$. Since $T$ is continuous, there exists a neighborhood $U$ of $r_0$ such that 
		$$T(s e^{i(\theta_t(r_0)+\varepsilon)}) > t, \quad \forall s \in U.$$
		Hence, $\theta_t(s) \leq \theta_t(r_0) + \varepsilon$, for all $s \in U.$ 
		
		In the case $\theta_t(r_0) =0$, we may then take $B=U$. In the case $\theta_t(r_0) \neq 0$, the monotonicity of $T$ shows that $T(r_0 e^{i(\theta_0(r_0)-\varepsilon)}) < t.$ Since $T$ is continuous, there exists a neighborhood $V$ of $r_0$ such that 
		$$T(s e^{i(\theta_0(r_0)-\varepsilon)}) < t, \quad \forall s \in V.$$
		Hence, $\theta_t(s) \geq \theta_t(r_0) - \varepsilon$,  for all $s \in V.$ Thus, we may take $B=U\cap V$.
	\end{proof}

	\begin{proof}[Proof of Proposition $\ref{openness}$.]
		By Proposition \ref{equidomain}, $\Sigma_t$ is the set $r_0 e^{i\theta_0}$ with $r_0\neq 0$ and $\vert\theta_0\vert<\theta_t(r_0)$. Fix $r_0e^{i\theta_0}\in\Sigma_t$. If $\theta_t(r_0)=\pi^+$, then by Point (1) of Lemma \ref{conttheta} there is a neighborhood $B$ of $r_0$ on which $\theta_t=\pi^+$, in which case all $s e^{i\theta}$ with $s\in B$ belong to $\Sigma_t$.
		
		If $\theta_t(r_0)\neq\pi^+$, set $\varepsilon=(\theta_t(r_0)-\vert\theta_0\vert)/2$. Then by Points (2) and (3) of Lemma \ref{conttheta}, there is a neighborhood $B$ of $r_0$ on which $$\theta_t(s)>\theta_t(r_0)-\varepsilon=\theta_0+\varepsilon.$$ Then all $s e^{i\theta}$ with $s\in B$ and $\theta$ within $\varepsilon$ of $\theta_0$ are in $\Sigma_t$.
	\end{proof}
	
	\begin{figure}[t]
	\centering
	\includegraphics[scale=0.6]{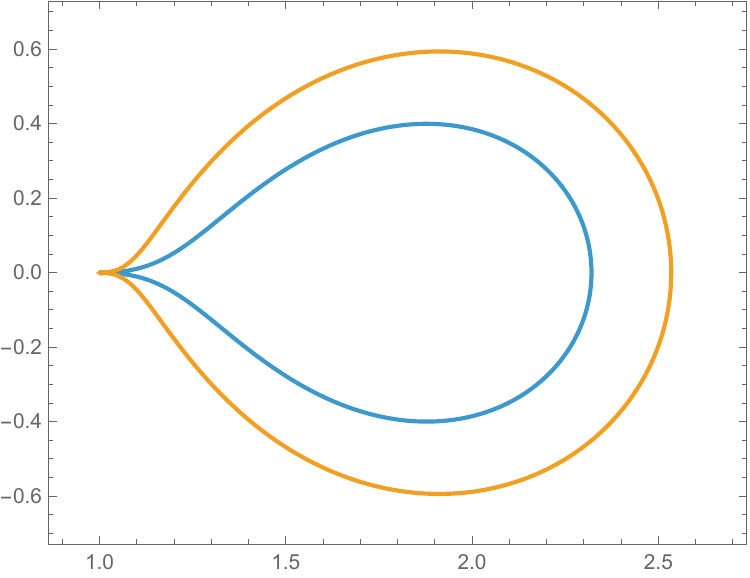}
	\caption{The domain $\overline\Sigma_t $ with $d\mu(\xi)=1_{[1,2]}(\xi-1)^2/3\,d\xi$ for $t=1/2$ (blue) and $t=1$ (orange). The function $T$ has a value of approximately 1.91 at the point 1, which is on the boundary of both domains}
	\label{cusp.fig}
	\end{figure}
	
	\begin{prop}\label{boundary}
	For nonzero $\lambda_0=r_0 e^{i\theta_0} \in \partial\Sigma_t$, we have
		\begin{itemize}
			\item[(1)] $\theta_t(r_0) \neq \pi^+$ and $\lambda_0 = r_0 e^{i\theta_t(r_0)}$,
			\item[(2)] $T(\lambda_0) \geq t$,
			\item[(3)] if $\theta_0 \neq 0$ or $\theta_0=0$ but $r_0\notin\supp(\mu)$, then $T(\lambda_0) =t$.
		\end{itemize}
		Furthermore, if  $\arg(\lambda_0) \neq 0$ and $T(\lambda_0) =t$, then $\lambda_0 \in \partial\Sigma_t$.   
	\end{prop}
	Figure \ref{cusp.fig} shows an example in which $T(\lambda_0)>t$ at a boundary point of $\Sigma_t$. In the example, $T(\lambda_0)=0$ for all $\lambda_0\in(1,2]$ but $T(1)>0$, because the momenta $\tilde p_0$ and $\tilde p_2$ from (\ref{pktildedef}) are infinite on $(0, 2]$ but finite at 1.
	\begin{remark}\label{tboundary.rem}
	If the condition (\ref{p0infinity}) in Remark \ref{allsupport.rem} holds, then $T\equiv 0$ on $\supp(\mu)\setminus\{0\}$, so that $\supp(\mu)\setminus\{0\}$ is contained in the open set $\Sigma_t$ for every $t>0$. In that case, we can conclude that $T(\lambda_0)=t$ for every nonzero $\lambda_0$ in the boundary of $\Sigma_t$.
	\end{remark}
	\begin{proof}
	Fix $\lambda_0=r_0 e^{i\theta_0}$, with $-\pi<\theta\leq\pi$, belonging to $\partial\Sigma_t\setminus\{0\}$ Then $\lambda_0$ cannot be in the open set $\Sigma_t$. Since $T(\bar\lambda_0)=T(\lambda_0)$, we can assume $\lambda_0$ is in the closed upper half plane. Then $\theta_t(r_0)$ cannot be $\pi^+$ or $\lambda_0$ would be in $\Sigma_t$. If $\theta_0=0$, then $\theta_t(r_0)$ must be zero, or $\lambda_0$ would be in $\Sigma_t$. If $0<\theta_0\leq\pi$, then $\theta_0$ cannot be less then $\theta_t(r_0)$ or $\lambda_0$ would be in $\Sigma_t$. But also $\theta_0$ cannot be greater than $\theta_t(r_0)$ or $\theta_t(r_0)$ would be less than $\pi$---and then by the continuity of $\theta_t$, there would be a neighborhood of $\lambda_0$ outside $\Sigma_t$. In that case, $\lambda_0$ could not be in $\partial\Sigma_t$. Thus, in all cases, $\theta_0=\theta_t(r_0)$, establishing the first part of the proposition. 
	
	For the second part, $T(\lambda_0)$ cannot be less than $t$ or $\lambda_0$ would be in $\Sigma_t$. For the third part, the assumptions ensure that $\lambda_0$ is not in $\supp(\mu)$, in which case,  $T$ is continuous at $\lambda_0$. Then since $\lambda_0$ is a limit of points with $T<t$, we have $T(\lambda_0)\leq t$ (and also $T(\lambda_0)\geq t$).		
	
		Lastly, let $\lambda_0$ be such that $T(\lambda_0) =t$ and $\arg(\lambda_0) \neq 0$. Then $\lambda_0$ is, by definition, not in $\Sigma_t$. But $\lambda_0$ must be in the closure of $\Sigma_t$, by the monotonicity of $T$ with respect to $\theta$. Thus, $\lambda_0$ must be in the boundary of $\Sigma_t$.	
	\end{proof}

	\subsection{The ``outside the domain'' calculation}
	In this section, we will always take $\lambda_0$ outside $\supp(\mu)$, in which case, the initial momenta in (\ref{prho})--(\ref{pepsilon}) will have finite limits as $\varepsilon_0$ tends to zero. We use the notation $$\mu_t=\text{ Brown measure of }xb_t.$$
	
	Our aim is to calculate
		\[s_t(\lambda) = \lim_{\varepsilon \rightarrow 0^+} S(t,\lambda,\varepsilon),\]
	using the Hamilton--Jacobi formula in Theorem $\ref{HJ}$. Our approach is to choose ``good'' initial conditions $\lambda_0$ and $\varepsilon_0$ in the Hamiltonian system (\ref{Ham1})--(\ref{Ham2}) (with initial momenta given by (\ref{prho})--(\ref{pepsilon})) so that $\lambda(t)=\lambda$ and $\varepsilon(t)=0$. 
	
	Now, in light of (\ref{epsilonoft}), we see that if $\varepsilon_0$ is very small, then $\varepsilon(s)$ will be positive but small for $0\leq s\leq t$, provided that the solution of the system (\ref{Ham1})--(\ref{Ham2}) exists up to time $t$. Furthermore, the argument $\theta=\arg(\lambda)$ is a constant of motion. Finally, by (\ref{loglog}), if $\varepsilon_0$ is small (so that $\varepsilon(s)$ is small for $s<t$), then $\vert\lambda(s)\vert$ will be close to $\vert\lambda_0\vert$ for $s<t$.
	
	We therefore propose to take $\lambda_0=\lambda$ and $\varepsilon_0\rightarrow 0$, with the result that
	\[\lambda(t) \equiv \lambda \quad \text{and} \quad \varepsilon(t) \equiv 0.\]
	Now we recall the Hamilton--Jacobi formula in (\ref{hjformula}), which reads (after using (\ref{h0})):
	\begin{equation}\label{HJrepeat}S(t,\lambda(t),\varepsilon(t)) = S(0,\lambda_0,\varepsilon_0) + \varepsilon_0p_0p_2t +\log|\lambda(t)| - \log|\lambda_0|.\end{equation}
	Suppose we can simply let $\varepsilon_0 \rightarrow 0$. Then the second term on the right-hand side of (\ref{HJrepeat}) will be zero for $\lambda_0\notin\supp(\mu)$, because $p_0p_2$ remains finite. Furthermore, since $\lambda(t)=\lambda_0$, the last two terms in the Hamilton--Jacobi formula cancel, leaving us with 
	\[
	S(t,\lambda,0)=S(0,\lambda,0)=\int_0^{\infty}\log(|\xi-\lambda|^2)\,d\mu(\xi).
	\]
	
	If the preceding formula holds on any open set $U$ outside the support of $\mu$, then $S(0,\lambda,0)$ will be harmonic there, meaning that the Brown measure $\mu_t$ is zero on $U$.
		We note, however, that in order for the preceding strategy to work, the lifetime of the solution to (\ref{Ham1})--(\ref{Ham2}) must remain greater than $t$ in the limit $\varepsilon_0\rightarrow 0$. We may then hope to carry out the strategy for $\lambda_0$ outside the closed domain $\overline\Sigma_t$, as the following theorem confirms.
	\begin{theorem}\label{main1}
		Fix a pair $(t,\lambda)$ with $\lambda$ outside $\overline{\Sigma}_t$. Then 
		$$ s_t(\lambda) := \lim_{\varepsilon_0 \rightarrow 0^+} S(t,\lambda,\varepsilon)= \int_0^{\infty}\log(|\xi-\lambda|^2)\,d\mu(\xi).$$
		In particular, the Brown measure $\mu_t$ of $xb_t$ is zero outside of $\overline{\Sigma}_t$ except possibly at the origin. Furthermore, if  $0\notin \overline{\Sigma}_t$, then
		\[\mu_t(\{0\}) = \mu(\{0\}).\]
	\end{theorem}
	The theorem tells us that the log potential of the Brown measure $\mu_t$ agrees outside $\Sigma_t$ with the log potential of the law $\mu$ of the self-adjoint element $x$. (This does not, however, mean that $\mu_t=\mu$.)
	
	We now begin working toward the proof of Theorem \ref{main1}. If $\lambda_0$ is outside $\Sigma_t$, then by definition, $T(\lambda_0)$ must be at least $t$. But if $T(\lambda_0)=t$ then by the last part of Proposition \ref{boundary}, $\lambda_0$ must be in the boundary of $\Sigma_t$, provided that $\lambda_0$ is not on the positive real axis. That is, $T(\lambda_0)>t$ for all $\lambda_0 \notin \compconj{\Sigma}_t \cup [0, \infty)$. Thus, for such $\lambda_0$, the lifetime $t_\ast(\lambda_0,\varepsilon_0)$ will be greater than $t$ for all sufficiently small $\varepsilon_0$. For $\lambda_0 \notin \compconj{\Sigma}_t \cup [0, \infty)$ and $\varepsilon_0 >0$,
	we define for each $t>0$ a map $U_t$ by 
	\begin{equation}\label{utdef} U_t(\lambda_0,\varepsilon_0) = (\lambda(t;\lambda_0,\varepsilon_0),\varepsilon(t;\lambda_0,\varepsilon_0)).\end{equation}
	where $\lambda(t;\lambda_0,\varepsilon_0)$ and $\varepsilon(t;\lambda_0,\varepsilon_0)$ denote the $\lambda$- and $\varepsilon$-components of the solution of (\ref{Ham1})--(\ref{Ham2}) with the given initial conditions. We note by Theorem \ref{t*} that this map makes sense even if $\varepsilon_0$ is slightly negative.
	
	We will evaluate the derivative of this map at $(\lambda_0,\varepsilon_0) = (\lambda_0,0)$.
	\begin{lemma}\label{jacobian} For all $\lambda_0\notin \compconj{\Sigma}_t \cup [0, \infty)$,
		the Jacobian of $U_t$ at $(\lambda_0,0)$ has the form
		$$ U_t^{'}(\lambda_0,0) = \begin{pmatrix}
			I_{2\times2} & \frac{\partial \lambda}{\partial \varepsilon_0}(t;\lambda_0,0)  \\
			0 & \frac{\partial \varepsilon}{\partial \varepsilon_0}(t;\lambda_0,0)
		\end{pmatrix}$$
		with $\frac{\partial \varepsilon}{\partial \varepsilon_0}(t;\lambda_0,0) > 0.$
	\end{lemma}
	\begin{proof} (Similar to Lemma 6.3 in \cite{dhk})
		Note that if $\varepsilon_0 = 0$, then $\varepsilon(t) \equiv 0$ and $\lambda(t) \equiv \lambda_0$. Thus, $U_t(\lambda_0,0) = (\lambda_0,0)$. Then we only need to show that $\frac{\partial \varepsilon}{\partial \varepsilon_0}(t;\lambda_0,0) > 0$. From Lemma \ref{pepsilon formula}, we have
		\begin{align*}
			\frac{\partial \varepsilon}{\partial \varepsilon_0}(t;\lambda_0,0) &= \frac{1}{p_\varepsilon(t)^2} p_0^2e^{-Ct} + \varepsilon_0 \left.\frac{\partial}{\partial \varepsilon_0}\left[\frac{1}{p_\varepsilon(t)^2} p_0^2e^{-Ct}\right]\right\vert_{\varepsilon_0 = 0}\\
			&=\frac{1}{p_\varepsilon(t)^2} p_0^2e^{-Ct} >0,
		\end{align*}
		as claimed.
	\end{proof}
	
	Now we are ready to prove our main result.
	\begin{proof}[Proof of Theorem \ref{main1}] 
		We first establish the result for $\lambda \notin \compconj{\Sigma}_t \cup [0, \infty)$. We take $\lambda_0=\lambda$ in Lemma \ref{jacobian}. Then by the inverse function theorem, the map $U_t$ in (\ref{utdef}) has a local inverse near $U_t(\lambda,0)=(\lambda,0)$. 
		
		Now, the inverse of the matrix in Lemma \ref{jacobian} will have a positive entry in the bottom right corner; that is, $U_t^{-1}$ has the property that $\partial\varepsilon_0/\partial\varepsilon >0$. Thus, the $\varepsilon_0$-component of $U_t^{-1}(\lambda,\varepsilon)$ will be positive for $\varepsilon$ small and positive. In that case, the solution of the Hamiltonian system (\ref{Ham1})--(\ref{Ham2}) will have $\varepsilon(u)>0$ up to the blow-up time. In turn, the blow-up time with initial conditions $(\lambda_0,\varepsilon_0)=U_t^{-1}(\lambda,\varepsilon)$ will exceed $t$ for $\varepsilon$ close to 0. 
		
	We now let ``HJ'' denote the quantity on the right-hand side of the Hamilton--Jacobi formula (\ref{HJrepeat}):
	\begin{equation}\label{hjfunction} \mathrm{HJ}(t,\lambda_0,\varepsilon_0) = S(0,\lambda_0,\varepsilon_0) + \varepsilon_0p_0p_2t +\log|\lambda(t)| - \log|\lambda_0| .\end{equation}
	If $\varepsilon$ is small and positive, the Hamilton--Jacobi formula (\ref{HJrepeat}) tells us that
	\begin{equation}\label{hjinverse} S(t,\lambda,\varepsilon) = \mathrm{HJ}(t,U_t^{-1}(\lambda,\varepsilon)).\end{equation}
	Now, the Hamilton--Jacobi formula is not directly applicable when $\varepsilon=0$, because $S(t,\lambda,\varepsilon)$ is defined only for $\varepsilon>0$.  But the map $U_t(\lambda_0,\varepsilon_0)$, defined in terms of the Hamiltonian system (\ref{Ham1})--(\ref{Ham2}), does makes sense when $\varepsilon<0$. Thus, the right-hand side of (\ref{hjinverse}) does make sense when $\varepsilon$ is zero or even slightly negative. Thus, (\ref{hjinverse}) provides a way of computing the \emph{limit} of $S(t,\lambda,\varepsilon)$ as $\varepsilon$ approaches zero. 
		
		In the limit $\varepsilon_0 \rightarrow 0^+$ with $\lambda$ fixed, the inverse function theorem tells us that $U_t^{-1}(\lambda,\varepsilon) \rightarrow (\lambda,0)$. Thus the limit of $S(t,\lambda,\varepsilon)$ as $\varepsilon$ tends to zero from above can be computed by taking $\lambda_0=\lambda$ and $\varepsilon_0=0$ on the right-hand side of (\ref{hjfunction}). Since $\lambda(t)$ becomes zero in this limit, as discussed at the beginning of this section, we get 
		\[\lim_{\varepsilon\rightarrow 0^+}S(t,\lambda,\varepsilon)=\lim_{\varepsilon_0\rightarrow 0^+,\,\lambda_0\rightarrow\lambda}S(0,\lambda_0,\varepsilon_0)=
		\int_0^\infty\log(|\xi-\lambda|^2)\,d\mu(\xi),
		\]
		where the second term on the right-hand side of (\ref{hjfunction}) tends to zero as $\varepsilon_0$ tends to zero, because $p_0$ and $p_2$ remain finite when $\lambda_0$ is outside $\supp(\mu)$. The limit of $S(t,\lambda,\varepsilon)$ as $\varepsilon$ tends to zero from above
		can be computed by putting $\lambda(t;\lambda_0,\varepsilon_0)= \lambda$ and letting $\varepsilon_0 \rightarrow 0$ and $\lambda_0 \rightarrow \lambda$ on the right-hand side of (\ref{hjinverse}). We therefore obtain the desired result when $\lambda$ is outside $\overline\Sigma_t\cup[0, \infty)$.
		
		We now have that
	\begin{equation}\label{s_t formula}
		s_t(\lambda) = \int_0^{\infty}\log(|\xi-\lambda|^2)\,d\mu(\xi)
	\end{equation}
		holds for any $\lambda \in \mathbb{C}$ outside $\overline{\Sigma}_t \cup [0, \infty)$. Since $s_t(\lambda)$ is subharmonic, it is it is locally integrable everywhere. (See Proposition \ref{limS.prop}.) Thus, $s_t$ can be interpreted as distribution by integrating against test functions with respect to Lebesgue measure on the plane. It follows that when computing $s_t$ in the distribution sense, we can ignore sets of Lebesgue measure zero, such as $[0, \infty)$. 
		
		We conclude, then, that the formula (\ref{s_t formula}) continues to hold in the distribution sense on the complement of $\overline{\Sigma}_t$. Now, the right-hand side of (\ref{s_t formula}) is a smooth function outside $\overline\Sigma_t\cup\{0\}$, by Lemma \ref{suppoutside}. Its Laplacian in the distribution sense is then the Laplacian in the classical sense, which can be computed by putting the Laplacian inside the integral, giving an answer of 0. Thus, $\mu_t$ is zero outside $\overline\Sigma_t\cup\{0\}$. Once this result is known, it is easy to see that (\ref{s_t formula}) actually holds in the pointwise sense outside $\overline\Sigma_t\cup\{0\}$.
		
		Finally, we compute the mass of $\mu_t$ at the origin, in the case that $0$ is outside $\overline\Sigma_t$. In that case, we can write
		\begin{align}\label{st split}
			s_t(\lambda) = \log(|\lambda|^2)\mu(\{0\}) + \int_{\supp(\mu)\setminus\{0\}} \log(|\xi-\lambda|^2)\,d\mu(\xi).
		\end{align}
		where $\supp(\mu)\setminus\{0\}$ is contained in $\overline\Sigma_t$ by Lemma \ref{suppoutside}. Thus, the second term on the right-hand side of (\ref{st split}) is harmonic outside $\overline\Sigma_t$. Thus, if $0 \notin \compconj{\Sigma}_t$, we obtain
		\[\frac{1}{4\pi}\Delta s_t = 
		\mu(\{0\})\delta_0 \quad \text{on $\compconj{\Sigma}_t^c$},\]
		as claimed.
	\end{proof} 
	
	\section{The general $\tau$ case}
	
	\subsection{Statement of results}
	
	Let $b_{s,\tau}$ be a three-parameter free multiplicative Brownian motion, as defined in Definition \ref{def b_st}, and let $x$ be a non-negative operator that is freely independent of $b_{s,\tau}$. Our goal is to compute the support of the Brown measure of $xb_{s,\tau}$. 
	
	Note that for a \emph{unitary} operator $u$ freely independent of $b_{s,\tau}$, the complement of the support of the Brown measure of $ub_{s,\tau}$ is obtained by mapping the complement of the support of $ub_{s,s}$ by a map $f_{s-\tau}$. (See Section 3 in \cite{hho}). This transformation $f_\alpha$ was introduced by Biane \cite[Section 4]{bianejfa} in the case where $u =1$. In our case, we will do something similar to get the complement of the support of the Brown measure of $xb_{s,\tau}$.
	
	Recall that $\mu$ is the law of the non-negative element $x$, as in (\ref{lawdef}).
	
	\begin{definition}\label{f.def}
		For $\alpha \in \mathbb{C}$, define a transformation $f_{\alpha}$ by 
		\begin{align*}
			f_{\alpha}(\lambda) = \lambda\exp\left[\frac{\alpha}{2}\int_0^{\infty} \frac{\xi+\lambda}{\xi-\lambda}\,d\mu(\xi)\right],
		\end{align*}
		where $\mu$ is the law of $x$ and $\lambda \notin \supp(\mu).$
	\end{definition}
	Note that as $\lambda$ tends to infinity, the exponential factor in $f_\alpha(\lambda)$ tends to $e^{-\alpha/2}$. Thus, $f_\alpha(\lambda)$ tends to infinity as $\lambda$ tends to infinity. 
	\begin{lemma}\label{fat0}
	Suppose $0$ is an isolated point in the support of $\mu$. Then $f_\alpha(\lambda)$ approaches $0$ as $\lambda$ approaches $0$. In this case, we can extend $f_\alpha$ to a holomorphic function on $(\mathbb C\setminus\supp(\mu))\cup\{0\}$ with $f_\alpha(0)=0$.
	\end{lemma}
	\begin{proof}
	Assume 0 is in $\supp(\mu)$ but $(0, c)$ is outside $\supp(\mu)$ for some $c>0$. Then for $\lambda\notin\supp(\mu)$, we have
	\[
	\int_0^\infty \frac{\xi+\lambda}{\xi-\lambda}\,d\mu(\xi)=-\mu(\{0\})+\int_c^\infty \frac{\xi+\lambda}{\xi-\lambda}\,d\mu(\xi)
	\]
	and the result follows easily.
	\end{proof}
	Recall from Corollary \ref{suppoutside} that $\supp(\mu)\setminus\{0\}$ is contained in $\overline\Sigma_s$. Thus, by Lemma \ref{fat0}, we can always define $f_\alpha$ as a holomorphic function on the complement of $\overline\Sigma_s$, even in the case where $0$ is outside $\overline\Sigma_s$ and $\mu$ has mass at 0. 
	\begin{definition}\label{domainD}
		For $s>0$ and $\tau \in \mathbb{C}$ such that $|s-\tau| \leq s$, define a closed set $D_{s,\tau}$ by the relation
		\[ (D_{s,\tau})^c = f_{s-\tau}((\overline{\Sigma}_s)^c)\]
		and
		\[\Sigma_{s,\tau} = \operatorname{int}(D_{s,\tau}), \]
		where we recall from Section 3 that \[\Sigma_s = \{\lambda : \lambda\neq 0 \text{ and } T(\lambda)<s\}.\]
	\end{definition}

	See Figure \ref{fig:second} for examples of the domains $D_{s,\tau}$.
	
	 We now define the regularized log potential of the Brown measure of $xb_{s,\tau}$, as follows.
	
	\begin{definition}\label{defS}
		Define a function $S$ by
		$$ S(s,\tau,\lambda,\varepsilon) = \tr\left[\log\left((xb_{s,\tau}-\lambda)^\ast(xb_{s,\tau}-\lambda)+\varepsilon^2\right)\right] $$
		for all $s>0, \tau \in \mathbb{C}$ with $|\tau -s | \leq s, \lambda\in \mathbb{C},$ and $\varepsilon >0$.
	\end{definition}
	Note that while in Section \ref{sequalstau.sec} we regularized the log potential of the Brown measure of $xb_t$ using $\varepsilon$, as in \cite{dhk} and \cite{hz}, here we regularize using $\varepsilon^2$ to be consistent with \cite{hho}. This convention will allow us to use formulas from \cite{hho} without change. 
	
	We now state the main result of this section, whose proof will occupy the rest of the section.
	\begin{theorem}\label{main2}
		Fix $s >0$ and $\tau \in \mathbb{C}$ such that  $|\tau-s|\leq s$. Then the following results hold.
		\begin{enumerate}
			\item The map $f_{s-\tau}$ is a bijection from $(\overline{\Sigma}_s)^c$ to $(D_{s,\tau})^c$ and $f_{s-\tau}(\lambda)$ tends to infinity as $\lambda$ tends to infinity.
			\item For all $\lambda$ outside of $D_{s,\tau}$, we have  \begin{equation}\label{dSz1}
				\frac{\partial S}{\partial \lambda} =   \frac{f^{-1}_{s-\tau}(\lambda)}{\lambda}\int_0^{\infty} \frac{1}{f^{-1}_{s-\tau}(\lambda) -\xi} \,d\mu(\xi).
			\end{equation}
			\item The Brown measure $\mu_{s,\tau}$ of $xb_{s,\tau}$ is zero outside $D_{s,\tau}$ except possibly at the origin.
			\item If 0 is not in $D_{s,\tau}$, we have
			 \[\mu_{s,\tau}(\{0\}) = \mu(\{0\}).\]
		\end{enumerate}
	\end{theorem}
Note that $f_{s-\tau}$ is surjective by Definition \ref{domainD}. The right-hand side of $(\ref{dSz1})$ is well-defined since the nonzero closed support of $\mu$ lies inside $\overline{\Sigma}_s$ by Corollary \ref{suppoutside}. This equation also implies that $\frac{\partial S}{\partial \lambda}$ is holomorphic function outside $D_{s,\tau}$ except possibly at the origin. Therefore, $(3)$ follows from $(2)$.
	\subsection{The PDE method}\label{pdeMethod}
	To obtain Theorem \ref{main2}, we need the following PDE and its Hamiltonian, obtained as Theorem 4.2 in \cite{hho}.
	\begin{theorem}[Hall--Ho]
		The function $S$ in Definition \ref{defS} satisfies the PDE 
		\begin{equation}\label{pdeS}
			\frac{\partial S}{\partial \tau} = \frac{1}{8}\left[1-\left(1-\varepsilon	\frac{\partial S}{\partial \varepsilon}-2\lambda	\frac{\partial S}{\partial \lambda}\right)^2\right]
		\end{equation}
		for all $\lambda\in\mathbb{C}, \tau \in \mathbb{C}$ satisfying $|\tau -s| < s,$ and $\varepsilon>0$ where
		 \[\frac{\partial}{\partial \tau} = \frac{1}{2}\left(\frac{\partial}{\partial \tau_1} - i\frac{\partial}{\partial \tau_2}\right) \quad \text{and } \quad \frac{\partial}{\partial \lambda}=\frac{1}{2}\left(\frac{\partial}{\partial \lambda_1} - i\frac{\partial}{\partial \lambda_2}\right)\] are the Cauchy--Riemann operators and the initial condition at $\tau =s$ is given by
		$$ S(s,s,\lambda,\varepsilon) = \tr\left[\log\left((xb_{s,s}-\lambda)^\ast(xb_{s,s}-\lambda)+\varepsilon^2\right)\right], $$
		where $b_{s,s}=b_s$ is the free multiplicative Brownian motion considered in Section 3. 
	\end{theorem}

	Moreover, from Section \ref{sequalstau.sec}, in the initial case $\tau =s$ we have that for $\lambda_0 \notin \compconj{\Sigma}_s$,
	\begin{equation}\label{S0}
	s(s,s,\lambda_0) := 	\lim_{\varepsilon \rightarrow 0^+} S(s,s,\lambda_0,0) = \int_0^{\infty} \log\left(|\xi-\lambda_0|^2\right)\,d\mu(\xi), 
	\end{equation}
	where $\mu$ is the law of the element $x$.
	
	We now introduce the complex-valued Hamiltonian function, obtained from the PDE (\ref{pdeS}) by replacing each derivative with a momentum variable, with an overall minus sign:
	$$ H(\lambda,\varepsilon,p_\lambda,p_\varepsilon) = -\frac{1}{8}\left[1-\left(1-\varepsilon p_\varepsilon -2\lambda p_\lambda\right)^2\right].$$
	Here the variables $\lambda$ and $p_\lambda$ are complex-valued and $\varepsilon$ and $p_\varepsilon$ are real-valued. For arbitrary $\lambda_0\in\mathbb C$ and $\varepsilon_0>0$, define initial momenta $p_{\lambda,0}$ and $p_{\varepsilon,0}$ by
	\begin{align}
		p_{\lambda,0} = \frac{\partial}{\partial \lambda}S(s,s,\lambda_0,\varepsilon_0)\label{4.4}\\
		p_{\varepsilon,0} = \frac{\partial}{\partial \varepsilon}S(s,s,\lambda_0,\varepsilon_0)\label{4.5}.
	\end{align}
		
	We now introduce curves $\lambda(\tau)$, $\varepsilon(\tau)$, $p_\lambda(\tau)$, and $p_\varepsilon(\tau)$ as in Section 5 in \cite{hho} but with $\tau$ there replaced by $\tau-s$ here, since our initial condition is at $\tau=s$:
	\begin{align}
		\lambda(\tau) &= \lambda_0\exp\left\{\frac{\tau-s}{2}\left(\varepsilon_0p_{\varepsilon,0} + 2\lambda_0p_{\lambda,0} - 1\right)\right\}\label{4.6}\\
		\varepsilon(\tau) &= \varepsilon_0\exp\left\{\Re\left[\frac{\tau-s}{2}\left(\varepsilon_0p_{\varepsilon,0} + 2\lambda_0p_{\lambda,0} - 1\right)\right]\right\}\label{4.7}\\
		p_\lambda(\tau) &= p_{\lambda,0}\exp\left\{-\frac{\tau-s}{2}\left(\varepsilon_0p_{\varepsilon,0} + 2\lambda_0p_{\lambda,0} - 1\right)\right\}\label{4.8}\\
		p_{\varepsilon}(\tau) &=p_{\varepsilon,0}\exp\left\{-\Re\left[\frac{\tau-s}{2}\left(\varepsilon_0p_{\varepsilon,0} + 2\lambda_0p_{\lambda,0} - 1\right)\right] \right\}\label{4.9}.
	\end{align}
	The initial values in all cases denote the values of the curves at $\tau=s$; for example, $\lambda_0$ is the value of $\lambda(\tau)$ at $\tau=s$. 
	
	Hall and Ho develop \cite[Section 4]{hho} Hamilton--Jacobi formulas for solutions to the PDE (\ref{pdeS}) by reducing the equations to an ordinary Hamilton--Jacobi PDE with a real time variable. These formulas hold initially for $\vert\tau-s\vert<s$ but extend by continuity to the boundary case $|s-\tau| = s$, as in \cite[Proposition 5.5]{hho}. We therefore obtain the following result.
	\begin{theorem}[Hall--Ho]\label{thm4.6}
		For all $\tau$ with $|\tau-s| \leq s,$ we have the first Hamilton--Jacobi formula
		\begin{equation} \label{HJ2}
			\begin{aligned}
			S(s,\tau,\lambda(\tau),\varepsilon(\tau)) &= S(s,s,\lambda_0,\varepsilon_0) + 2 \Re[(\tau-s) H_0] \\&+\frac{1}{2}\Re\left[(\tau-s)(\varepsilon_0p_{\varepsilon,0}+2\lambda_0p_{\lambda,0})\right],
			\end{aligned}			
		\end{equation}
		where $H_0 = H(\lambda_0,\varepsilon_0,p_{\lambda,0},p_{\varepsilon,0}),$ and the second Hamilton--Jacobi formulas
		\begin{align}
			\frac{\partial S}{\partial \lambda}(s,\tau,\lambda(\tau),\varepsilon(\tau)) &= p_\lambda(\tau)\label{spz}\\
			\frac{\partial S}{\partial \varepsilon}(s,\tau,\lambda(\tau),\varepsilon(\tau)) &= p_\varepsilon(\tau).
		\end{align}
	\end{theorem}
	
	\subsection{The $\varepsilon\rightarrow 0$ limit}
	
	We are then interested in the behavior of the curves in (\ref{4.6})--(\ref{4.9}) in the limit as $\varepsilon$ tends to zero.
	\begin{lemma}\label{zfz0} 
		
		Suppose that $\lambda_0$ is a nonzero point outside $\overline\Sigma_s$. Then for all $\tau$ with $|\tau-s| \leq s,$ the limits as  $\varepsilon_0\rightarrow 0$ of $p_{\varepsilon,0}$ and $\lambda_0p_{\lambda,0}$ exist and	\[\lim_{\varepsilon_0 \rightarrow 0}\lambda(\tau) = f_{s-\tau}(\lambda_0).\]
		The same result holds for $\lambda_0=0$, provided that $0\notin\overline\Sigma_s$ and $0\notin\supp(\mu)$.
	\end{lemma}
	\begin{proof}
		Let $\lambda_0 \notin \compconj{\Sigma}_s$ be nonzero. Then $p_{\varepsilon,0}$ is the same as $p_0$ in (\ref{pkdef}) in Section 3. Thus, by Corollary \ref{suppoutside}, $p_{\varepsilon,0}$ remains finite as $\varepsilon_0 \rightarrow 0$. Meanwhile, as $\varepsilon_0 \rightarrow 0$, we have
		\begin{align*}
			p_{\lambda,0} &= \frac{\partial}{\partial \lambda_0} S(s,s,\lambda_0,\varepsilon_0)\\
			&= \frac{\partial}{\partial \lambda_0} \int_0^{\infty} \log\left(|\xi-\lambda_0|^2\right)d\mu(\xi) \\
			&= \int_0^{\infty} \frac{1}{\lambda_0 -\xi} \,d\mu(\xi).
		\end{align*}
		It follows that
		\begin{align*}
			\lim_{\varepsilon_0 \rightarrow 0} 2\lambda_0p_{\lambda,0} -1 &= 2\lambda_0 \int_0^{\infty} \frac{1}{\lambda_0 -\xi} \,d\mu(\xi) -1\\
			&= - \int_0^{\infty} \frac{\xi+\lambda_0}{\xi-\lambda_0} \,d\mu(\xi).
		\end{align*}
		Hence, letting $\varepsilon$ tend to zero in the formula $(\ref{4.6})$ for $\lambda(\tau)$, we get
		\begin{align*}
			\lim_{\varepsilon_0 \rightarrow 0}\lambda(\tau) &= \lim_{\varepsilon_0 \rightarrow 0} \lambda_0\exp\left\{\frac{\tau-s}{2}\left(\varepsilon_0p_{\varepsilon,0} + 2\lambda_0p_{\lambda,0} - 1\right)\right\}\\
			&= \lambda_0\exp\left[\frac{\tau-s}{2}\left(-\int_0^{\infty} \frac{\xi+\lambda_0}{\xi-\lambda_0} \,d\mu(\xi)\right)\right] \\
			&= f_{s-\tau}(\lambda_0).
		\end{align*}
		The same calculation is applicable if $\lambda_0\notin\overline\Sigma_s$ equals 0, provided that 0 is not in the support of $\mu$.
	\end{proof}
	
	\begin{prop}\label{6.2} Define the map $\phi_\tau$ from $\mathbb{C} \times (0, \infty)$ into $\mathbb{C}\times (0, \infty)$ by 
		$$\phi_\tau(\lambda_0,\varepsilon_0) = (\lambda(\tau),\varepsilon(\tau))$$
		where $\lambda(\tau)$ and $\varepsilon(\tau)$ are computed with $\lambda(s) = \lambda_0$ and $\varepsilon(s) = \varepsilon_0$ and the initial momenta $p_{\lambda,0}$ and $p_{\varepsilon,0}$ given by (\ref{4.4}) and (\ref{4.5}). Then for all nonzero $\lambda_0 \in (\overline{\Sigma}_s)^c$, the map $\phi_\tau$ extends analytically to a neighborhood of $\varepsilon_0 =0$ as a map into $\mathbb C\times\mathbb R$. Moreover, Jacobian of $\phi_\tau$ at $(\lambda_0,0$) is invertible if and only if $f'_{s-\tau}(\lambda_0) \neq 0$. 
	\end{prop}
	\begin{proof}
		By Corollary \ref{suppoutside}, every nonzero  $\lambda_0$ in $(\overline{\Sigma}_s)^c$ is outside of the closed support of $\mu$. Then the formulas $(\ref{4.4})$ and $(\ref{4.5})$ for the initial momenta $p_{\lambda,0}$ and $p_{\varepsilon,0}$ are well defined and analytic even in a neighborhood of $\varepsilon_0 =0$. Thus the formulas $(\ref{4.6})$ and $(\ref{4.7})$ for $\lambda(\tau)$ and $\varepsilon(\tau)$ depend analytically on $\lambda_0$ and $\varepsilon_0$.
		
		Now, to compute the Jacobian of $\phi_\tau$ at $(\lambda_0,0)$, we write $\lambda_0 = x_0 +iy_0.$ Differentiating (\ref{4.6}) and evaluating at $\varepsilon_0 =0$ gives 
		$$ \frac{\partial \varepsilon}{\partial x_0} = 0;\quad \frac{\partial \varepsilon}{\partial y_0} = 0; \quad\frac{\partial \varepsilon}{\partial \varepsilon_0} = \exp\left(\Re\left[\frac{\tau-s}{2}(2\lambda_0p_{\lambda,0}-1)\right]\right)> 0.$$
		Meanwhile, when $\varepsilon_0 =0$, Lemma \ref{zfz0} tells us that $\lambda(\tau) = f_{s-\tau}(\lambda_0)$. Thus, Jacobian of $\phi_\tau$ has the form
		$$\begin{pmatrix}
			J & \ast  \\
			0 &  \frac{\partial \varepsilon}{\partial \varepsilon_0}
		\end{pmatrix}$$
		where $J$ is the Jacobian of $f_{s-\tau}(\lambda_0)$ at $\lambda_0$ (that is, the complex number $f_{s-\tau}'(\lambda_0)$, viewed as a $2\times 2$ matrix). Therefore the Jacobian of $\phi_\tau$ at $(\lambda_0,0)$ is invertible if and only if  $f'_{s-\tau}(\lambda_0) \neq 0$.
	\end{proof}
	\begin{prop}\label{6.3}
		Fix $\tau$ with $|\tau -s| \leq s$ and a nonzero point $w \in (D_{s,\tau})^c$. Choose $w_0 \in (\overline{\Sigma}_s)^c$ such that $f_{s-\tau}(w_0) = w$, where such a $w_0$ exists by Definition \ref{domainD}. If $f'_{s-\tau}(w_0) \neq 0$, then the map
		$$ (\lambda,\varepsilon) \mapsto S(s,\tau,\lambda,\varepsilon),$$
		initially defined for $\varepsilon >0$, has an analytic extension defined for $(\lambda,\varepsilon)$ in a neighborhood of $(w,0)$.
	\end{prop}
	\begin{proof}
		We let HJ$(s,\tau,\lambda_0,\varepsilon_0)$ denote the right-hand side of the first Hamilton--Jacobi formula (\ref{HJ2}), i.e.,
		\begin{multline*} 
		\text{HJ}(s,\tau,\lambda_0,\varepsilon_0) \\
		=S(s,s,\lambda_0,\varepsilon_0) + 2 \Re[(\tau-s) H_0] +\frac{1}{2}\Re\left[(\tau-s)(\varepsilon_0p_{\varepsilon,0}+2\lambda_0p_{\lambda,0})\right].
		\end{multline*}
		Fix a nonzero $w \in (D_{s,\tau})^c$ and pick $w_0 \in (\overline{\Sigma}_s)^c$ such that $f_{s-\tau}(w_0) =w$. By Lemma \ref{zfz0} with $\varepsilon_0 =0$, we have $\lambda(\tau)= f_{s-\tau}(w_0) =w$. By Proposition \ref{6.2}, the map $\phi_\tau$ has a local inverse near $(w_0,0)$, provided that $f_{s-\tau}'(w_0)\neq 0$. 
		
		We may therefore define a function $\tilde S$ by
		$$ \tilde{S}(s,\tau,\lambda,\varepsilon) = \text{HJ}(s,\tau,\phi_\tau^{-1}(\lambda,\varepsilon))$$
		Then $\tilde{S}$ agrees with $S$ as long as the first Hamilton--Jacobi formula (\ref{HJ2}) is applicable, namely for $\varepsilon_0 >0$. Thus, $\tilde{S}$ will be the desired analytic extension, provided that $f'_{s-\tau}(w_0) \neq 0$.
	\end{proof}
	We now want to let $\varepsilon_0$ tend to zero in the second Hamilton--Jacobi formula (\ref{spz}). We do this at first for points $\lambda$ outside $D_{s,\tau}$ for which we can find $\lambda_0$ outside $\overline\Sigma_s$ with $f_{s-\tau}(\lambda_0)=\lambda$ and $f_{s-\tau}'(\lambda_0)\neq 0$.
	\begin{prop}\label{main}
		Fix $\tau$ with $|\tau -s| \leq s$ and a nonzero point $\lambda \in (D_{s,\tau})^c$. Choose $\lambda_0 \in (\overline{\Sigma}_s)^c$ such that $f_{s-\tau}(\lambda_0) = \lambda$, and assume that $f'_{s-\tau}(\lambda_0) \neq 0$. Then the analytic extension of $S$ given by Proposition \ref{6.3} satisfies
		\begin{equation}\label{dSz}
			\frac{\partial S}{\partial \lambda}(s,\tau,\lambda,0) =   \frac{f^{-1}_{s-\tau}(\lambda)}{\lambda}\int_0^{\infty} \frac{1}{f^{-1}_{s-\tau}(\lambda) -\xi} \,d\mu(\xi)
		\end{equation}
		where the inverse is taken locally. 
	\end{prop}
	\begin{proof}
		Since $f'_{s-\tau}(\lambda_0)\neq 0$, it directly follows from Proposition $\ref{6.3}$ that we can let $\varepsilon_0 \rightarrow 0$ in the second Hamilton--Jacobi formula $(\ref{spz})$. Now, if we compare the formulas (\ref{4.6}) and (\ref{4.8}), we see that 
		\[
		p_\lambda(\tau)=p_{\lambda,0}\frac{\lambda_0}{\lambda(\tau)}.
		\]
		Furthermore, by Lemma \ref{zfz0}, we have (with $\varepsilon=0$) that $\lambda(\tau)= f_{s-\tau}(\lambda_0)=\lambda$. Thus, the analytic extension of $S$ satisfies 
		\begin{align*}
			\frac{\partial S}{\partial \lambda}(s,\tau,\lambda,0) &= \frac{\lambda_0}{\lambda}p_{\lambda,0}\\
			&= \frac{\lambda_0}{f_{s-\tau}(\lambda_0)}\int_0^{\infty} \frac{1}{\lambda_0 -\xi} \,d\mu(\xi) \\
			&= \frac{f^{-1}_{s-\tau}(\lambda)}{\lambda}\int_0^{\infty} \frac{1}{f^{-1}_{s-\tau}(\lambda) -\xi} \,d\mu(\xi),
		\end{align*}
		where the inverse is locally defined and holomorphic by the holomorphic version of the inverse function theorem. 
	\end{proof}
	
	\begin{prop}\label{brownzero.prop}
		 The Brown measure of $xb_{s,\tau}$ is zero outside $D_{s,\tau}$, except possibly at the origin. 
	\end{prop}
	\begin{proof}
	For all nonzero $w\in(D_{s,\tau})^c$, Proposition \ref{main} exhibits $\partial S/\partial \lambda (s,\tau,\lambda,0)$ near $w$ as a holomorphic function, provided that there is $\lambda_0\notin\overline\Sigma_s$ with $f_{s-\tau}(\lambda_0)=\lambda$ and $f_{s-\tau}'(\lambda_0)\neq 0$. Then Brown measure of $xb_{s,\tau}$, obtained by taking a $\bar \lambda$derivative, is therefore zero near any such $\lambda$.
	
	Now suppose that $\lambda\neq 0$ is outside $D_{s,\tau}$ and that for all $\lambda_0\notin \overline\Sigma_s$ such that $f_{s,\tau}(\lambda_0)=\lambda$, we have $f_{s,\tau}'(\lambda_0)=0$. Then any one such $\lambda_0$ will be an isolated zero of $f_{s,\tau}$. By the open mapping theorem and the result in the previous paragraph, the Brown measure of $xb_{s,\tau}$ will be zero in a punctured neighborhood of $\lambda$. If the Brown measure assigned positive mass to ${\lambda}$ (but is zero in a punctured neighborhood of $\lambda$), then $S(s,\tau,\lambda,0)$ would have a logarithmic singularity at $\lambda$ and $\partial S/\partial \lambda(s,\tau,\lambda,0)$ would blow up at $\lambda$. But if we evaluate $(\ref{dSz})$ at $w$ and let $w$ tend to $\lambda$, we get 
		\begin{align*}
			\lim_{w \rightarrow \lambda} \frac{\partial }{\partial \lambda}S(s,\tau,w,0) = \lim_{w_0 \rightarrow \lambda_0}\frac{w_0}{f_{s-\tau}(w_0)}\int_0^{\infty} \frac{1}{w_0 -\xi} \,d\mu(\xi)
		\end{align*}
		which remains bounded.
		\end{proof}
	\begin{prop}\label{masszero.prop}
	If $0$ is not in $D_{s,\tau}$, the Brown measure $\mu_{s,\tau}$ of $xb_{s,\tau}$ satisfies
	 \[\mu_{s,\tau}(\{0\}) = \mu(\{0\}).\]
	\end{prop}
	
	\begin{proof}
		 Note that since $f_{s-\tau}(\lambda_0)$ has a simple zero at $\lambda_0 =0$, we can construct a local inverse $f_{s,\tau}^{-1}$ near 0 having a simple zero at 0. Then $f^{-1}_{s-\tau}(\lambda)/\lambda$ has a removable zero at $\lambda=0$. Also, recall from Corollary \ref{suppoutside} that if 0 is outside $\overline\Sigma_s$, it is an isolated point in $\supp(\mu)$. In that case, we can find $c>0$ such that $(0, c) \cap \supp(\mu) = \varnothing$ and (\ref{dSz}) becomes
		\begin{equation}\label{4.14}
			\frac{\partial S}{\partial \lambda}(s,\tau,\lambda,0) = \frac{\mu(\{0\})}{\lambda} +  \frac{f^{-1}_{s-\tau}(\lambda)}{\lambda}\int_c^\infty \frac{1}{f^{-1}_{s-\tau}(\lambda) -\xi} \,d\mu(\xi).
		\end{equation} 
		Since the second term on the right-hand side of (\ref{4.14}) is holomorphic near the origin, we can take the distributional $\lambda$-bar derivative to obtain the distributional Laplacian, giving 
		\[ \mu_{s,\tau}(\{0\})= \mu(\{0\}) \]
		as claimed.
	\end{proof}
	
	\begin{prop}\label{finject.prop}
	For all $s$ and $\tau$ with $\vert s-\tau\vert\leq s$, the following results hold.
	
	\begin{enumerate}
	\item The function $f_{s-\tau}$ is injective on the complement of $\overline\Sigma_s$.
	\item The quantity $f_{s-\tau}(\lambda_0)$  tends to infinity as $\lambda_0$ tends to infinity.
	\item The set $D_{s,\tau}$ is compact.
	\item Assume that $\vert\tau-s\vert<s$, that the condition (\ref{p0infinity}) in Remark \ref{allsupport.rem} holds, and that $0\notin\partial\Sigma_s$. Then $f_{s-\tau}$ is defined and injective on the complement of $\Sigma_s$.

	\end{enumerate}
	
	\end{prop}
	
	Point 4 of the proposition says that (under the stated assumptions), the injectivity of $f_{s-\tau}$ in Point 1 extends to the boundary of $\Sigma_s$. The assumption $\vert\tau-s\vert<s$ in Point 4 cannot be omitted. Consider, for example, the case $\tau=0$ (so that $\vert\tau-s\vert=s$) and $x=1$. In that case, the map $f_{s-\tau}=f_s$ has already been studied in \cite{bianejfa} and \cite{dhk} and maps the boundary of $\Sigma_s$ in a generically two-to-one fashion to the unit circle. Specifically, points on the boundary with the same argument but different radii have the same value of $f_s$. See Figure \ref{2domains1.fig} in Section \ref{tauzero.sec} along with the discussion surrounding Proposition 2.5 in \cite{dhk}.
	
	\begin{proof}
	Now that the Brown measure of $xb_{s,\tau}$ is known (Proposition \ref{brownzero.prop}) to be zero outside $D_{s,\tau}$, we know that $\partial S/\partial \lambda(s,\tau,\lambda,0)$ is holomorphic outside $D_{s,\tau}$. It is then possible to use Proposition \ref{main} to show that $f_{s,\tau}'$ is nonzero outside $D_{s,\tau}$, showing that $f_{s,\tau}$ is \emph{locally} injective.  The argument is that if $f_{s,\tau}'(\lambda_0)=0$, then $\partial S/\partial \lambda$ would blow up at $f_{s,\tau}(\lambda_0)$. We omit the details because we will prove global injectivity by a different method. 
		
	To obtain global injectivity, we use an argument similar to Proposition 5.1 in \cite{pz}, by considering the second Hamilton--Jacobi formula (\ref{spz}). The argument is briefly as follows. We know from Proposition \ref{6.3} that $S(s,\tau,\lambda,\varepsilon)$ has an analytic extension to $\varepsilon=0$, for $\lambda\in(D_{s,\tau})^c$. If there were two different initial conditions $\lambda_0$ and $\tilde \lambda_0$ giving the same value of $f_{s-\tau}$---and thus, by Lemma \ref{zfz0}, the same value of $\lambda(\tau)$ at $\varepsilon_0=0$---we would get two different values for $\partial S/\partial \lambda(s,\tau,\lambda,0)$ at the same $\lambda$. 
	
	Filling in the details, we use the notation $$s(s,\tau,\lambda)=S(s,\tau,\lambda,0).$$ We first note that, in the case $0\notin\overline\Sigma_s$, we have from Definition \ref{f.def} that $f_{s-\tau}(\lambda_0)=0$ if and only if $\lambda_0=0$. 
	Using Proposition \ref{6.3}, we can let $\varepsilon$ tend to 0 in the second Hamilton--Jacobi formula (\ref{spz}). We therefore obtain		
		\begin{align*}
			\lambda(\tau)\frac{\partial s}{\partial \lambda}\left(\lambda(\tau),\tau\right)&= \lambda(\tau)p_\lambda(\tau) = \lambda_0 p_{\lambda,0},
		\end{align*}
		where the second equality follows from the explicit formulas (\ref{4.6}) and (\ref{4.8}).
		
		Now, if two initial conditions $\lambda_0$ and  $\tilde{\lambda}_0$ give the same value of $f_{s-\tau} = \lambda(\tau)$, then
		
		\begin{equation}\label{z0pz0}
		\lambda_0p_{\lambda,0} = \lambda(\tau)\frac{\partial s}{\partial \lambda}\left(\lambda(\tau),\tau\right) =\tilde{\lambda_0} \tilde{p}_{\lambda,0}.
		\end{equation}
		But then by the calculation in the proof of Lemma \ref{zfz0}, we have
		\begin{align*}
		\lambda_0\exp\left[\frac{\tau-s}{2}\left(2\lambda_0 p_{\lambda,0}-1\right)\right]  &=f_{s-\tau}(\lambda_0) \\
		&= f_{s-\tau}(\tilde{\lambda}_0) \\
		&= \tilde{\lambda}_0\exp\left[\frac{\tau-s}{2}\left(2\tilde{\lambda}_0 \tilde{p}_{\lambda,0}-1\right)\right].
		\end{align*}
		But since, by (\ref{z0pz0}), $\lambda_0p_{\lambda,0}=\tilde \lambda_0\tilde p_{\lambda,0}$, we must have $\lambda_0=\tilde \lambda_0$.
				
		Next, we have that 
		\[\lim_{\lambda_0\rightarrow \infty} \int_0^{\infty} \frac{\xi+\lambda_0}{\xi-\lambda_0} \,d\mu(\xi) = -1.\]
		That is 
		\begin{align*}
			\lim_{\lambda_0\rightarrow \infty} f_{s-\tau}(\lambda_0) = \lim_{\lambda_0\rightarrow \infty} \lambda_0 \exp\left[\frac{\tau-s}{2}\right] = \infty,
		\end{align*}
		as claimed. 
		
		Thus, $f_{s-\tau}$ extends to a holomorphic function on a neighborhood of $\infty$ on the Riemann sphere, mapping $\infty$ to $\infty$. Then by the open mapping theorem, the image of $f_{s-\tau}$ contains a neighborhood of $\infty$. Since, also, $\Sigma_s$ is bounded by Point 3 of Proposition \ref{valuet}, we see that $D_{s,\tau}$ (as defined in Definition \ref{domainD}) is a closed and bounded (i.e., compact) set.
		
		We now address the last point in the proposition. If the condition (\ref{p0infinity}) holds, then by Remark \ref{allsupport.rem}, $T\equiv 0$ on $\supp(\mu)\setminus\{0\}$, so that $\supp(\mu)\setminus\{0\}$ is contained in $\Sigma_r$ for all $r>0$. Thus, by Lemma \ref{fat0}, the domain of definition of $f_{s-\tau}$ contains the complement of $\Sigma_s$. Now, since $\vert\tau-s\vert<s$, we have
		\[
		\vert(\tau-\varepsilon)-(s-\varepsilon)\vert=\vert\tau-s\vert<s-\varepsilon
		\]
		for sufficiently small $\varepsilon$. Thus, we can apply Point 1 of the proposition with $(s,\tau)$ replaced by $(s-\varepsilon,\tau-\varepsilon)$ to conclude that $f_{s-\tau}=f_{(s-\varepsilon)-(\tau-\varepsilon)}$ is injective on $(\overline\Sigma_{s-\varepsilon})^c$. But Remark \ref{tboundary.rem} tells us that every nonzero point in $\overline\Sigma_{s-\varepsilon}$ is in $\Sigma_s$ (and thus, not in $\partial\Sigma_s$), under the condition (\ref{p0infinity}). Thus, assuming $0\notin\partial\Sigma_s$, we find that injectivity of $f_{s-\tau}$ extends to the boundary of $\Sigma_s$.
	\end{proof}
	
	\begin{proof}[Proof of Theorem \ref{main2}]
	The theorem follows from Propositions \ref{main}, \ref{brownzero.prop}, \ref{masszero.prop}, and \ref{finject.prop}.
	\end{proof}

	\section{The boundary case $\tau =0$}\label{tauzero.sec}
	
	Recall from Definition \ref{def b_st} that a free multiplicative Brownian motion $b_{s,\tau}$ is defined for $|s-\tau|\leq s$. In this section, we focus on the borderline case of $\tau =0$. Recall also from the discussion after Definition \ref{def b_st} that $b_{s,0}$ is the free unitary Brownian motion $U_s$ considered by Biane in \cite{bia}. Thus we study $xb_{s,0}$ where $x$ is non-negative operator and freely independent of $b_{s,0}$. The case in which $x=cI$ is special in that $xb_{s,0}$ is $c$ times a unitary operator, so that the Brown measure of $xb_{s,0}$ is supported on a circle of radius $c$ centered at the origin.
	
Note that the domain $D_{s,0}$ from Definition \ref{domainD} when $\tau =0$ is defined by the relation
	\[
	(D_{s,0})^c = f_{s}((\overline{\Sigma}_s)^c)\]
With $\tau=0$, results from Theorem \ref{main2} can be stated as follows.

	\begin{theorem}
		For a fixed $s > 0$, we have
		\begin{itemize}
			\item The map $f_{s}: (\overline{\Sigma}_s)^c \rightarrow (D_{s,0})^c$ is injective.
			\item The Brown measure $\mu_{s,0}$ of $xb_{s,0}$ is zero outside $D_{s,0}$ except possibly at the origin.
		\end{itemize}
	\end{theorem}

	Demni and Hamdi \cite{dem} carry out a PDE analysis for the regularized Brown measure of $xb_{s,0}$, culminating in a formula \cite[Theorem 1.1]{dem} analogous to our Theorem \ref{thm4.6}. In determining the support of the Brown measure in, however, they specialize to the case in which $x$ is a self-adjoint projection. See Theorem 1.2 in \cite{dem}, where the authors are computing the Brown measure of the operator $xb_{s,0}x$ in the compressed von Neumann algebra $x\mathcal A x$, which is the same as the Brown measure of $xb_{s,0}$ in $\mathcal A$, after removing an atom at the origin and multiplying by a constant. Our results therefore generalize \cite{dem} by treating arbitrary non-negative $x$ and arbitrary $\tau$ with $\vert \tau-s\vert\leq s$.
	
	Demni and Hamdi first identify a domain $\Sigma_{t,\alpha}$, where $\alpha$ is the trace of the projection $x$, which they show is bounded by a Jordan curve. Then they define another domain $\Omega_{t,\alpha}$ whose boundary is the image of the boundary of $\Sigma_{t,\alpha}$ under a map $f_{t,\alpha}$. Now, the map $f_{t,\alpha}$ in \cite[Theorem 1.2]{dem} is easily seen to be what we call $f_t$. Meanwhile, the domain $\Sigma_{t,\alpha}$ in \cite[Proposition 3.2]{dem} is defined by the condition $T_\alpha(\lambda_0)<t$. Thus, if we can verify that their function $T_\alpha$ agree with our $T$, then we will see that their $\Sigma_{t,\alpha}$ equals our $\Sigma_t$ and thus that their $\Omega_{t,\alpha}$ is the interior of our $D_{t,0}$. Thus, their Theorem 1.2 will agree with Point 3 of our Theorem \ref{main2}. (Recall that Demni and Hamdi remove the mass of the Brown measure at the origin.)
			
	\begin{figure}[t]
	\centering
	\includegraphics[scale=0.5]{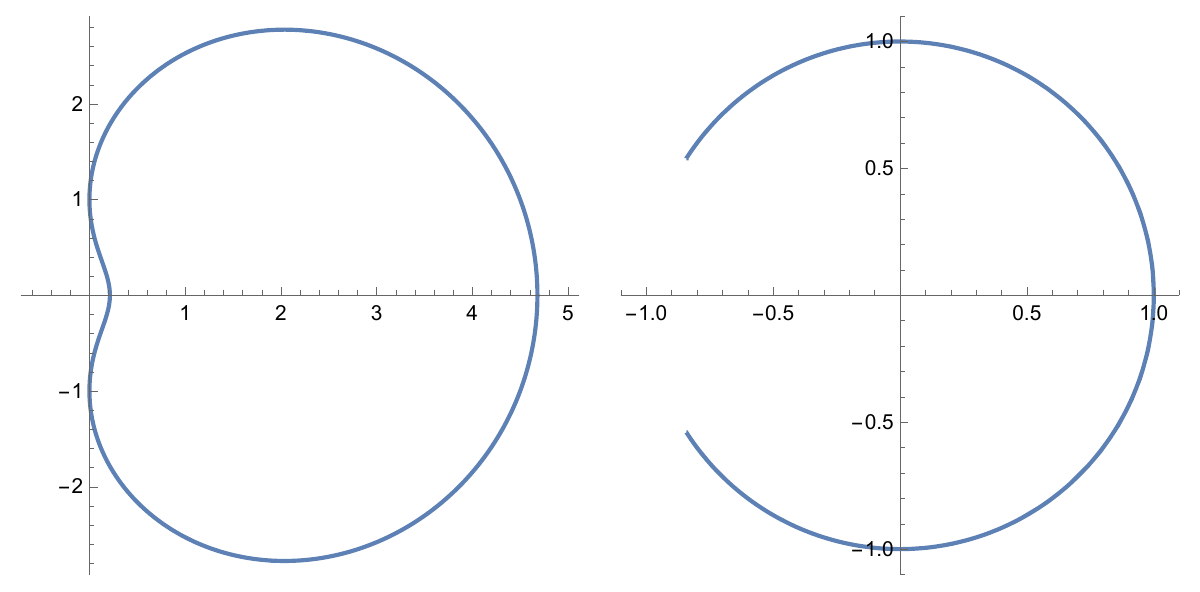}
	\caption{The domains $\overline\Sigma_s$ (left) and $D_{s,0}$ (right) with $s=2$ and $\mu=\delta_1$}
	\label{2domains1.fig}
	\end{figure}
	
	\begin{figure}[t]
	\centering
	\includegraphics[scale=0.5]{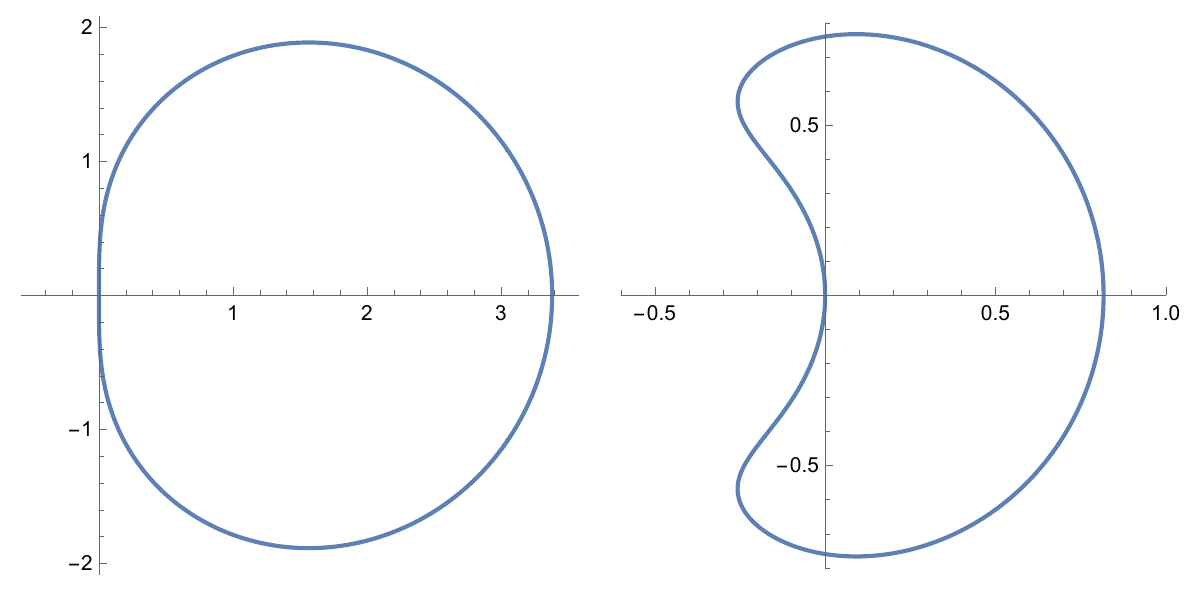}
	\caption{The domains $\overline\Sigma_s$ (left) and $D_{s,0}$ (right) with $s=2$ and $\mu=0.5\,\delta_0+0.5\,\delta_1$}
	\label{2domains2.fig}
	\end{figure}

	As depicted in Figures \ref{2domains1.fig} and \ref{2domains2.fig}, the domain $D_{s,0}$ collapses onto the unit circle in the case when $x$ is the identity operator (i.e., $\mu=\delta_1$), while no such collapse occurs when $x$ is not a multiple of the identity (i.e., $\mu$ is not concentrated at a single point).
	
	\begin{prop}
		The function $T_\alpha$ in \cite[Proposition 2.10]{dem} agrees with our function $T$, when $x$ is a self-adjoint projection with trace $\alpha$. 
	\end{prop}
	 \begin{proof}	 	
	 	By Proposition 2.10 in \cite{dem}, we have that 
	 	
	 	\begin{align*}
	 		T_\alpha &= \frac{1}{2\dot{v}(0)}\log\left(1+ \frac{2\dot{v}(0)}{r_0^2\tau(q_0)}\right) \\
	 		&= \frac{1}{1-\tilde{p}_\rho}\log\left(1+ \frac{1-\tilde{p}_\rho}{r_0^2\tilde{p}_0}\right)
	 	\end{align*}
	 	where the first line uses the notation in \cite{dem} and the second line uses our notation from Section 3, where $\tilde{p}_\rho$ is the value of  $p_\rho$ at $\varepsilon_0=0$. Thus we only need to show that 
	 	\[1-\tilde p_\rho = \tilde{p}_2 - \tilde{p}_0r_0^2.\]
	 	But from (\ref{prho}), we have
	 	
	 	\begin{align*}
	 		1- \tilde{p}_\rho &= \tr\left[\frac{|x-\lambda_0|^2 - 2r^2_0 + 2xr_0\cos\theta}{|x-\lambda_0|^2}\right] \\
	 		&= \tr\left[\frac{x^2 - 2xr_0\cos\theta + r_0^2 - 2r^2_0 + 2xr_0\cos\theta}{|x-\lambda_0|^2}\right] \\
	 		&= \tilde{p}_2 - \tilde{p}_0r_0^2.
	 	\end{align*}
		and the proof is complete.
	 \end{proof}
	 
	 \subsection*{Acknowledgments} The authors thank the referee for a careful reading of the paper, which has led to several substantial improvements. 
	
	\bibliographystyle{abbrv} 
	\bibliography{mybib} 

\begin{thebibliography}{10}

\bibitem{Azarin}
V.~Azarin.
\newblock {\em Growth theory of subharmonic functions}.
\newblock Birkh\"auser Advanced Texts: Basler Lehrb\"ucher. [Birkh\"auser
  Advanced Texts: Basel Textbooks]. Birkh\"auser Verlag, Basel, 2009.

\bibitem{banna}
M.~Banna, M.~Capitaine, and G.~Cébron.
\newblock Strong convergence of multiplicative brownian motions on the general
  linear group, 2025.

\bibitem{bia}
P.~Biane.
\newblock Free {B}rownian motion, free stochastic calculus and random matrices.
\newblock In {\em Free probability theory ({W}aterloo, {ON}, 1995)}, volume~12
  of {\em Fields Inst. Commun.}, pages 1--19. Amer. Math. Soc., Providence, RI,
  1997.

\bibitem{bianejfa}
P.~Biane.
\newblock Segal--{B}argmann transform, functional calculus on matrix spaces and
  the theory of semi-circular and circular systems.
\newblock {\em J. Funct. Anal.}, 144(1):232--286, 1997.

\bibitem{BianeSpeicher}
P.~Biane and R.~Speicher.
\newblock Stochastic calculus with respect to free {B}rownian motion and
  analysis on {W}igner space.
\newblock {\em Probab. Theory Related Fields}, 112(3):373--409, 1998.

\bibitem{brown}
L.~G. Brown.
\newblock Lidski\u i's theorem in the type {${\rm II}$} case.
\newblock In {\em Geometric methods in operator algebras ({K}yoto, 1983)},
  volume 123 of {\em Pitman Res. Notes Math. Ser.}, pages 1--35. Longman Sci.
  Tech., Harlow, 1986.

\bibitem{chan}
A.~Z. Chan.
\newblock The {S}egal-{B}argmann transform on classical matrix {L}ie groups.
\newblock {\em J. Funct. Anal.}, 278(9):108430, 59, 2020.

\bibitem{dem}
N.~Demni and T.~Hamdi.
\newblock Support of the {B}rown measure of the product of a free unitary
  {B}rownian motion by a free self-adjoint projection.
\newblock {\em J. of Funct. Anal.}, 282(6):109362, 2022.

\bibitem{driverSBT}
B.~K. Driver.
\newblock On the {K}akutani-{I}t\^o-{S}egal-{G}ross and
  {S}egal-{B}argmann-{H}all isomorphisms.
\newblock {\em J. Funct. Anal.}, 133(1):69--128, 1995.

\bibitem{dhk}
B.~K. Driver, B.~Hall, and T.~Kemp.
\newblock The {B}rown measure of the free multiplicative {B}rownian motion.
\newblock {\em Prob. Theory Related Fields}, 184(1-2):209--273, 2022.

\bibitem{sevenAuthors}
B.~K. Driver, B.~C. Hall, C.~W. Ho, T.~Kemp, Y.~Nemish, E.~A. Nikitopoulos, and
  F.~Parraud.
\newblock Matrix random walks and the lima bean law, 2025.

\bibitem{dhk_largeN}
B.~K. Driver, B.~C. Hall, and T.~Kemp.
\newblock The large-{$N$} limit of the {S}egal-{B}argmann transform on
  {$\mathbb{U}_N$}.
\newblock {\em J. Funct. Anal.}, 265(11):2585--2644, 2013.

\bibitem{evans}
L.~C. Evans.
\newblock {\em Partial differential equations}, volume~19 of {\em Graduate
  Studies in Mathematics}.
\newblock American Mathematical Society, Providence, RI, second edition, 2010.

\bibitem{hall1}
B.~C. Hall.
\newblock The {S}egal-{B}argmann ``coherent state'' transform for compact {L}ie
  groups.
\newblock {\em J. Funct. Anal.}, 122(1):103--151, 1994.

\bibitem{newform}
B.~C. Hall.
\newblock A new form of the {S}egal-{B}argmann transform for {L}ie groups of
  compact type.
\newblock {\em Canad. J. Math.}, 51(4):816--834, 1999.

\bibitem{qmbook}
B.~C. Hall.
\newblock {\em Quantum theory for mathematicians}, volume 267 of {\em Graduate
  Texts in Mathematics}.
\newblock Springer, New York, 2013.

\bibitem{hho}
B.~C. Hall and C.-W. Ho.
\newblock The {B}rown measure of a family of free multiplicative {B}rownian
  motions.
\newblock {\em Probab. Theory Related Fields}, 186(3-4):1081--1166, 2023.

\bibitem{hk}
B.~C. Hall and T.~Kemp.
\newblock Brown measure support and the free multiplicative {B}rownian motion.
\newblock {\em Advances in Mathematics}, 355:106771, 2019.

\bibitem{hoSBT}
C.-W. Ho.
\newblock The two-parameter free unitary {S}egal-{B}argmann transform and its
  {B}iane-{G}ross-{M}alliavin identification.
\newblock {\em J. Funct. Anal.}, 271(12):3765--3817, 2016.

\bibitem{hz}
C.-W. Ho and P.~Zhong.
\newblock Brown measures of free circular and multiplicative {B}rownian motions
  with self-adjoint and unitary initial conditions.
\newblock {\em J. Eur. Math. Soc. (JEMS)}, 25(6):2163--2227, 2023.

\bibitem{Hor}
L.~H\"ormander.
\newblock {\em The analysis of linear partial differential operators. {I}}.
\newblock Classics in Mathematics. Springer-Verlag, Berlin, 2003.
\newblock Distribution theory and Fourier analysis, Reprint of the second
  (1990) edition.

\bibitem{kemp}
T.~Kemp.
\newblock The large-{$N$} limits of {B}rownian motions on {$\mathbb{GL}_N$}.
\newblock {\em Int. Math. Res. Not. IMRN}, (13):4012--4057, 2016.

\bibitem{mingo}
J.~A. Mingo and R.~Speicher.
\newblock {\em Free probability and random matrices}, volume~35 of {\em Fields
  Institute Monographs}.
\newblock Springer, New York; Fields Institute for Research in Mathematical
  Sciences, Toronto, ON, 2017.

\bibitem{vaki}
E.~A. Nikitopoulos.
\newblock It\^o's formula for noncommutative {$C^2$} functions of free
  {I}t\^o{} processes.
\newblock {\em Doc. Math.}, 27:1447--1507, 2022.

\bibitem{pz}
P.~Zhong.
\newblock Brown measure of the sum of an elliptic operator and a free random
  variable in a finite von {N}eumann algebra.
\newblock {\em arXiv:2108.09844v4}, 2022.

\end{thebibliography}
\end{document}